\documentclass[12pt]{amsart}
\usepackage[english]{babel}
\usepackage{amssymb, amsmath}
\usepackage{txfonts}
\usepackage{calrsfs}
\usepackage{mathrsfs}
\usepackage{amscd}
\usepackage{graphicx}
\usepackage[applemac]{inputenc} 
\usepackage[T1]{fontenc} 
\usepackage{hyperref}
\usepackage{cite}
\usepackage{environ}
\usepackage{xcolor}

\DeclareMathOperator*{\holim}{holim} %* subscript under
\DeclareMathOperator*{\hocolim}{hocolim} 
\DeclareMathOperator{\Spec}{Spec}

\DeclareMathOperator{\Hom}{Hom}
\DeclareMathOperator{\uHom}{R\underline{Hom}} 
\DeclareMathOperator{\Aut}{Aut}

\DeclareMathOperator{\End}{End}
\DeclareMathOperator{\uEnd}{R\underline{End}}

\input diagxy

\begin{document}

\markright{A derived period map}
\bibliographystyle{hsiam}

\setcounter{tocdepth}{2} 
\setlength{\parindent} {0pt}
\setlength{\parskip}{1ex plus 0.5ex}

%some category notation
\newcommand{\cat}[1]{\mathscr{#1}} % dto outside $$
\newcommand{\mcat}[1]{$\mathscr{#1}$} % dto outside $$
\newcommand{\hocat}[1]{Ho\mathscr{#1}} % hoC
\newcommand{\ob}{\textrm{ob }}
\newcommand{\mor}{\textrm{mor }}
\newcommand{\id}{\mathbf 1} % the identity

%some categories
\newcommand{\sSet}{\mathbf{sSet}}
\newcommand{\Alg}{\mathbf{Alg}}
\newcommand{\sPr}{\mathbf{sPr}}
\newcommand{\dgCat}{\mathbf{dgCat}} 
\newcommand{\Vect}{\mathbf{Vect}} 
\newcommand{\Ch}{\mathbf{Ch}} 
\newcommand{\Chp}{\mathbf{Ch}_{pe}} 
\newcommand{\Cat}{\mathbf{Cat}} 
\newcommand{\sCat}{\mathbf{sCat}}
\newcommand{\qCat}{\mathbf{qCat}}
\newcommand{\stn}{\mathbf{Stn}}
\newcommand{\staCat}{\mathbf{Cat}^{Ex}_\oo}

\newcommand{\Sym}{\textrm{Sym}}

\newcommand{\mods}{\textrm{-Mod}}
\newcommand{\Map}{\textrm{Map}}
\newcommand{\Tor}{\textrm{Tor}}
\newcommand{\Ext}{\textrm{Ext}}
\newcommand{\Kos}{\textrm{Kos}} 

\newcommand{\Sing}{\textrm{Sing}} 

\newcommand{\G}{\mathbb{G}} 

\newcommand{\margin}[1]{\marginpar{\footnotesize #1}}

%sets
\newcommand{\set}[1]{\mathbb{#1}}
\newcommand{\Q}{\mathbb{Q}}
\newcommand{\C}{\mathbb{C}}
\newcommand{\Z}{\mathbb{Z}}
\newcommand{\R}{\mathbb{R}}

%geometry
\newcommand{\PR}{\mathbb{P}}
\newcommand{\OO}{\mathscr{O}} 
\newcommand{\MM}{\mathscr{M}} 
\newcommand{\A}{\mathscr{A}} 
\newcommand{\B}{\mathscr{B}} 

% greeks
\newcommand{\De}{\Delta}
\newcommand{\Ga}{\Gamma}
\newcommand{\Om}{\Omega}
\newcommand{\ep}{\epsilon}
\newcommand{\de}{\delta}
\newcommand{\la}{\lambda}
\newcommand{\al}{\alpha}
\newcommand{\om}{\omega}

%other stuff
\renewcommand{\to}{\rightarrow}
\newcommand{\ra}{\rightarrow}
\newcommand{\we}{\tilde \ra} %xrightarrow{\sim}}
\newcommand{\cof}{\hookrightarrow}
\newcommand{\fib}{\twoheadrightarrow}
\newcommand{\acof}{\tilde \hookrightarrow}
\newcommand{\afib}{\tilde \twoheadrightarrow}

\newcommand{\oo}{\infty}

\newcommand{\op}{^{\textrm{op}}} 

\newcommand{\Bold}{\boldsymbol}

\theoremstyle{plain}
\newtheorem{thm}{Theorem}[section]
\newtheorem{cor}[thm]{Corollary}
\newtheorem{lemma}[thm]{Lemma}
\newtheorem{propn}[thm]{Proposition}
\newtheorem{conj}{Conjecture}
\newtheorem{claim}{Claim}

\makeatletter
\newtheorem*{rep@theorem}{\rep@title}
\newcommand{\newreptheorem}[2]{
\newenvironment{rep#1}[1]{
 \def\rep@title{#2 \ref{##1}}
 \begin{rep@theorem}}
 {\end{rep@theorem}}}
\makeatother

\newtheorem{theorem}{Theorem}
\newreptheorem{theorem}{Theorem}

\theoremstyle{definition}
\newtheorem*{defn}{Definition}
\newtheorem*{altdef}{Alternative Definition}
\newtheorem{eg}{Example}
\newtheorem*{conv}{Convention}
\newtheorem{fact}{Construction}
\newtheorem*{qn}{Question}

\theoremstyle{remark}
\newtheorem{rk}[thm]{Remark}

\newtheorem{ork}{Temporary Remark}

\def \frk{\color{gray}\begin{ork}}
\def \endfrk{\end{ork}\color{black}}

\NewEnviron{killcontents}{}

\title{The global derived period map}
\author{Carmelo Di Natale}
\address{Carmelo Di Natale, King's College London}
\email{carmelo.di\_natale@kcl.ac.uk}
\author{Julian V. S. Holstein}
\thanks{The second author was supported by EPSRC grant EP/N015452/1 for part of this work}
\address{Julian V. S. Holstein, St John's College, University of Cambridge}
\email{julianvsholstein@gmail.com}
 
\maketitle

\begin{abstract}
We develop the global period map in the context of derived geometry, generalising Griffiths' classical period map as well as the infinitesimal derived period map. We begin by constructing the derived period domain which classifies Hodge filtrations and enhances the classical period domain. We analyze the monodromy action. Then we associate to any polarized smooth projective map of derived stacks a 
canonical morphism of derived analytic stacks from the base into the quotient of the derived period domain by monodromy. We conclude the paper by discussing a few examples and a derived Torelli problem. 
In the appendix we describe how to present derived analytic Artin stacks as hypergroupoids, which may be of independent interest. 
\end{abstract}

\tableofcontents{}

\section{Introduction}

We construct a global derived period map, generalising Griffiths' period map \cite{Griffiths70} and the infinitesimal derived period map of deformation theories \cite{Fiorenza08, Fiorenza09, diNatale14}.

We hope that this map will be a useful tool in studying derived moduli of varieties.

We work over $\C$ throughout this article. Given a polarized smooth projective map $f: X \to S$ between derived Artin stacks we construct a map $\cat P$ from the analytification $S^{an}$ of the base to a quotient of a derived analytic stack $U$ which we call the \emph{derived period domain}. The underived truncation 
 of $\cat P$ induces a product of the usual period maps on coarse moduli spaces; in particular a closed point $s \in S$ is sent to the Hodge filtration on the fibre $X_{s}$. Moreover, $\cat P$ extends the infinitesimal period map.

Our first result constructs the target of the derived period map:

\begin{reptheorem}{thm-perioddomain}
There is a derived period domain $U$ which is a geometric derived analytic Artin stack that extends the classical period domain. 
\end{reptheorem}

We define $U$ as an open subspace of the analytification of a derived stack $D_{n}(V,Q)$ which classifies filtrations of a complex $V$ that satisfy the Hodge-Riemann bilinear relations with respect to a shifted bilinear form $Q$. (We use analytification rather than a direct construction in analytic stacks because the theory of derived analytic stacks is not yet as developed as that for derived algebraic stacks.) The main ingredients are Lurie's derived analytification functor with its universal property \cite{Lurie11e, Lurie11d}, Porta's further analysis of this functor including its extension to derived Artin stacks \cite{Porta15a}, 
and some explicit constructions of stacks building on work of To\"en and Vezzosi \cite{Toen05a}.

In order to define maps into this stack we use the universal property for derived analytification of affine schemes together with a presentation of derived (analytic) Artin stacks as simplicial derived affine schemes, resp.\ derived Stein spaces. This is closely related to the theory presenting stacks as hypergroupoids that was developed by Pridham \cite{Pridham09a}, and that we extend to analytic stacks in the appendix.

We then take the quotient of $U$ by an arithmetic group $\Ga$ containing the fundamental group of the underlying topological space of $S^{an}$.
We use Deligne's work on formality \cite{Deligne68} to show that we need not look at the action of the full simplicial loop group of the base but only the fundamental group. 

We can now state our main theorem, which is proved using local computations in derived algebraic geometry, a derived trace map, classical Hodge theory and some homotopy theory: 
\begin{reptheorem}{mainA}
Let  $f: X \to S$ be a any polarized smooth projective map between derived Artin stacks where $S$ is connected and of finite presentation. Then there is a derived period map $\cat P: S^{an}\to U/\Ga$ of derived analytic stacks. 
\end{reptheorem}
It then follows from the construction that 
$\cat P$ enhances the classical period map and the infinitesimal derived period map. 

\begin{rk} 
The study of generalised period maps goes back to work in the infinitesimal setting motivated by mathematical physics. 
In order to study mirror symmetry for a Calabi-Yau manifold $X$ Barannikov and Kontsevich considered the extended formal moduli space $\cat M$ associated to the dg Lie algebra $\cat A^{0,*}(X, \wedge^{*} \cat T_{X})$. They construct a \emph{generalized period map} $\cat M \to \oplus_{k} H^{k}(X, \C)[\dim X-k]$ and use it to give $H^*(X, \wedge^{*}\cat T_{X})$ the structure of a formal Frobenius algebra which is related to the Gromov-Witten invariants of the mirror partner of $X$ \cite{Barannikov98,Barannikov00}. See also \cite{Barannikov99}.
\end{rk}

\subsection{Outline}

We begin by recalling in Section \ref{sect-background} some basic notions of derived 
algebraic and derived analytic geometry, in particular the derived analytification functor connecting them, which was defined by Lurie and further described by Porta. 

In Section \ref{sect-perdomain} we construct the derived period domain $U$. After recalling the classical period domain in Section \ref{sect-classicaldomain} we review the derived flag variety in Section \ref{sect-alggrass}. This derived stack classifies filtrations of a complex $V$. For us $V$ will be the cohomology of a fibre of a smooth projective morphism of derived stacks. In Section \ref{sect-algdomain} we build a geometric derived stack $D_n(V,Q)$ classifying filtrations on a complex $V$ equipped with a bilinear form $Q$, satisfying the Hodge-Riemann orthogonality relation. 
This is a closure of the derived period domain.
We compute the tangent space in Section \ref{sect-tangent}.
In Section \ref{sect-anadomain} we construct the derived period domain $U$ itself using the analytification of $D_{n}(V,Q)$ and the characterization of open derived substacks in terms of open substacks of the underlying underived stack. We later use the universal property of analytification to construct maps into this space.
We verify $U$ is a derived enhancement of the usual period domain.

In Section \ref{sect-monodromy} we recall Deligne's result that the derived pushforward of the constant sheaf for a family of smooth projective varieties is formal. We deduce there is no higher monodromy acting on cohomology of the fibre. We then construct the quotient of the derived period domain by the action induced by the monodromy action on $V$.

In Section \ref{sect-map} we construct the derived period map. The heart of the matter is in Sections \ref{sect-localperiod-1} and \ref{sect-localperiod-2}, where we construct the derived period map locally on contractible derived Stein spaces. As in the underived case we push forward the relative cotangent complex; the stupid filtration gives the Hodge filtration and we obtain the polarization by extending the trace pairing to derived geometry.  Using the universal property of analytification we construct a map to the derived period domain. 
In Section \ref{sect-globalperiod} we glue the local period maps on a simplicial resolutions of $S$ and the construction is complete. The universal example is given by the moduli stack of smooth polarized schemes. 
We check that our map is an enhancement of the usual period map in Section \ref{sect-comparison}. In Section \ref{sect-derivative} we compute the differential and show that the derived period map extends the infinitesimal derived period map.
We briefly talk about examples in Section \ref{sect-examples}.

In Appendix \ref{appendix} we develop the theory of hypergroupoids in derived Stein spaces as a model for derived analytic Artin stacks. 

\subsection{Notation and conventions}
We work over the field of complex numbers $\set C$ throughout. Tensor products are understood over $\C$, respectively the constant sheaf $\underline {\set C}$, unless otherwise indicated.

We will use the following notation: $\sSet$ is the category of simplicial sets with the classical (Quillen) model structure; $\sCat$ is the category of categories enriched over simplicial sets; $\dgCat$ is the category of differential graded categories;
$\mathbf{sAlg}$ is the category of simplicial commutative algebras (over $\C$). 

We write $\Map(-,-)$ for the mapping space in a simplicial category or in a model category, in a dg category  $\Hom(-,-)$ denotes the dg enriched hom space. If there is an underlying model category structure we write $\uHom(-,-)$ for the derived hom complex. We use cohomological grading throughout.

Given a derived scheme or derived analytic space $X$ over $\C$ we denote by $t(X)$ the underlying topological space. For a (derived) scheme $V$ this is $V(\C)$ with the classical topology.

A constant simplicial object is denoted by $c$, i.e. $(cA)_{n}=A$ for all $n$. 
For a derived Stein space $T$ we let $uT$ denote the derived affine scheme $R\Spec(\cat O^{alg}(T))$.
We write $*$ for the point, i.e.\ the derived scheme $\Spec(\C)$ and its analytification.

\subsection{Acknowledgements}

We would like to thank Jon Pridham and Ian Grojnowski for introducing the first, respectively second, author to this problem, and for many helpful discussions.
We are grateful to Mauro Porta for generous explanations of his work on derived analytic geometry, particularly for providing the proof of Lemma \ref{lem-gaga}, and for many insightful discussions about this paper. In particular he pointed out a crucial mistake in the first version of this paper. Discussions of different possible solutions led to the joint work \cite{HolsteinF}.
We would like to thank Barbara Fantechi, Benjamin Hennion, Daniel Huybrechts, Peter J\o rgensen and Bertrand To\"en for useful conversations. We also thank Serguei Barannikov for pointing out interesting references and the anonymous referees for many very helpful comments and corrections.

\section{Derived geometry} \label{sect-background}

\subsection{Derived algebraic geometry}

We will be assuming some familiarity with derived algebraic geometry, which is algebraic geometry that is locally modelled on the model category of simplicial commutative algebras instead of the category of commutative algebras. 

There is a vast literature on the subject, developed by To\"en-Vezzosi, Lurie and Pridham among others. For an introduction see \cite{Toen05, Toen14}.

Here we just mention a few reminders and conventions.

We will be studying certain \emph{derived stacks}. Just like a scheme can be represented by a set-valued sheaf on commutative algebras a derived stack can be represented by a simplicial-set-valued hypersheaf on simplicial commutative algebras. (We always take stack to mean higher stack.) In fact derived stacks can be described by a model structure on presheaves of simplicial sets on simplicial commutative algebras. 
A representable derived stack is also called a derived affine stack, and the image of a simplicial algebra $A$ in derived stacks is denoted $R\Spec(A)$. 

We will call a derived stack \emph{Artin} or simply \emph{geometric} if there exist certain smooth covers. Here we use geometric to mean $k$-geometric for some $k$ rather than 1-geometric.
The precise definition is inductive on $k$, beginning with representable stacks, which are $(-1)$-geometric. See for example Definition 1.3.3.1 in \cite{Toen05a}.
Replacing smooth by \'etale covers one obtains the definition of a \emph{Deligne-Mumford stack} (we will call them DM stacks for short). A derived stack that is a union of geometric ones is called \emph{locally geometric}. 
It is often interesting and consequential to show that certain derived stacks arising as moduli functors, say, are geometric.

Derived stacks come with simplicial algebras of functions $\cat O$ and a representable 
map $f: X \to Y$ between derived stacks is called \emph{strong} if $\pi_{i}(\cat O_{X}) \simeq \pi_{i}(f^{-1}\cat O_{Y})\otimes_{\pi_{0}(f^{-1}\cat O_{Y})} \pi_{0}(\cat O_{X})$ for $i > 0$.

One fact we will use frequently is that $k$-geometric derived stacks are stable under homotopy pullbacks in derived stacks, Corollary 1.3.3.5 in \cite{Toen05a}.

As we work over $\C$ the normalization functor $N$ from simplicial commutative algebras to commutative dg-algebras concentrated in non-positive degrees is a Quillen equivalence and we will freely switch between the two viewpoints.

One can restrict a derived stack to an underived stack along $\mathbf{Alg} \to \mathbf{sAlg}$. We denote this operation by $\pi^{0}$ (it is often written $t_{0}$ in the literature).

On the other hand one can truncate a stack to a functor into sets, by applying $\pi_{0}$ object-wise. We denote this functor by $\pi_{0}$. This construction has the universal property of the coarse moduli space of the stack.

We will often consider the double truncation $\pi_{0}\pi^{0}$ and if $\pi_{0}\pi^{0}(\cat X) = X$ for some scheme $X$ we will say the derived stack $\cat X$ is an \emph{enhancement} of $X$. (Here we deviate a little from the literature, where one typically considers enhancements of higher stacks, i.e.\ a derived stack $\cat X$ is an enhancement of the higher stack $X$ if $\pi^{0}(\cat X) = X$.)

A derived affine schemes is \emph{almost of finite presentation} 
if the functor it represents commutes with filtered (homotopy) colimits. 
A derived stack is \emph{locally almost of finite presentation} if it has a cover consisting of derived affine schemes almost of finite presentation. 
It follows from Theorem 7.4.3.18 of \cite{Lurie11}
that a derived stack $X$ with perfect cotangent complex and such that $\pi^{0}(X)$ is finitely presented is locally almost of finite presentation.

We will be considering derived stacks taking values in simplicial sets. In moduli questions these often come from functors valued in $(\oo,1)$-categories, and several of our constructions are clearest in terms of categories rather than simplicial sets. Our preferred models for $(\oo,1)$-categories are dg-categories and simplicial categories. We use the right Quillen functors $DK: \dgCat \to \sCat$ obtained by composing truncation, the Dold-Kan construction and the forgetful functor (see \cite{Tabuada10}), to move from dg-categories to simplicial categories. (We will sometimes use the same name for associated sheaves of dg-categories and simplicial categories).
To associate a simplicial set we take the maximal subgroupoid $Int(\cat C)$ and then use
the functor
$\bar W: \sCat \to \sSet$ defined in Definition 1.6 of \cite{Pridham10}. It is weakly equivalent to taking the diagonal of the nerve. The composition of $\bar W \circ Int \circ DK$ is also weakly equivalent to the maximal subgroupoid of the dg nerve in Construction 1.3.1.6 of \cite{Lurie11}.  
The reader is welcome to think of either of these functors instead.
Abusing notation slightly we write $N_{W}$ for both $\bar W \circ Int$ and $\bar W\circ Int \circ DK$. 

$DK$ preserves homotopy limits and $\bar W$ preserves homotopy pullbacks of simplicial categories which are homotopy groupoids, see Proposition 1.8 in \cite{Pridham10}.

\begin{rk}
To reassure the reader that the construction $\bar W$ is the natural one we can note that if we begin with a simplicial model category $\cat M$ we can also consider its classifying space, which is just the nerve of the subcategory of weak equivalences. It follows from Corollary 1.10 in \cite{Pridham10} and Proposition 8.7 in \cite{Rezk01} that this simplicial set is weakly equivalent to $\bar W \cat M$.
\end{rk}

\begin{rk}
The reader should be advised that there are different definitions of (geometric) derived stacks. We use results by To\"en-Vezzosi and Lurie respectively, which a priori live in different frameworks. The framework we have used above is To\"en-Vezzosi's, where we understand a derived stack as its functor of points.

We will also need to refer to Lurie's approach using structured topoi. A description can be found in the next section.
These are shown to give equivalent models for Deligne-Mumford stacks in \cite{Porta15b}. It is worth noting that there is no direct approach to Artin stacks using structured topoi.

The most concrete approach to derived algebraic geometry is in terms of hypergroupoids. This theory is developed by Pridham in \cite{Pridham09a}, where the equivalence with To\"en-Vezzosi's definitions of Artin stacks is established.

This approach is based on noting that all geometric stacks can be presented as simplicial affine schemes satisfying some technical condition. It may also be extended to the analytic setting, see Appendix \ref{appendix}.

There is also some care needed as there are differences in terminology, cf. Remark 2.12 in \cite{Pridham10}, but the reader can safely ignore these differences unless she or he cares about the precise value of $k$ for which a geometric stack is $k$-geometric.
\end{rk}

\subsection{Derived analytic geometry following Lurie and Porta}\label{sect-dan}

Since the period map is a priori holomorphic and not algebraic we have to construct its derived enhancement in the setting of derived analytic geometry. The theory of derived analytic geometry is still being developed by Lurie, Porta, and Ben-Bassat and Kremnizer. 

The main difficulty is that while derived algebraic geometry is modelled on simplicial algebras, derived analytic geometry should be modelled on simplicial analytic algebras, but it is not obvious what those should be. 
One approach, using Ind-Banach algebras, is developed in \cite{Ben-Bassat13}. 

The approach we will use is Lurie's theory of structured spaces, see \cite{Lurie11e} and Sections 11 and 12 of \cite{Lurie11d}, in particular as extended by Porta in his work around derived GAGA \cite{Porta15, Porta15a}. Porta and Yu have also worked out derived analytic deformation theory \cite{Porta16}.

We will use this work largely as a black box, providing us with a good theory of derived analytic DM stacks. In the following we will give some vague explanations while giving references for precise definitions and results. We recommend that the interested reader turn to the introduction of \cite{Porta15a} for further explanations.

The crucial object is a category $\cat Top(\cat T_{an})$ of $\cat T_{an}$-structured topoi. We can think of an object as an $\oo$-topos  $\cat X$ together with a sheaf of simplicial commutative rings $\cat O_{\cat X}^{alg}$ and some extra structure. (In the first approximation the reader may think of $\cat X$ as a topological space and ignore the extra structure on $\cat O_{X}$.) %\margin{2.4}

\begin{rk}
Derived analytic spaces have more structure than simply a sheaf of simplicial commutative algebras because one needs to keep track of the action of holomorphic functions on subsets of $\C^{n}$ by post-composition. 
In the differentiable setting Spivak's definition of simplicial $\cat C^{\oo}$-rings is motivated by this issue, see \cite{Spivak10}.
\end{rk}

The objects of interest to us form the subcategory $dAn_{\C}$ of derived analytic spaces, see Definition 1.3 in \cite{Porta15a}. They are the enhancements of complex analytic spaces. They also contain a subcategory of derived Stein spaces, which we will denote by $dStein$. (It is called $Stn^{der}_{\C}$ in loc.\ cit.)

Together with its analytic topology and the collection of smooth morphisms $dStein$ forms a geometric context in the sense of Porta-Yu \cite{Porta14}. That means that to define \emph{derived analytic Artin stacks} one can consider the category of simplicial sheaves on $dStein$ and use the notion of geometric stacks in the sense of \cite{Porta14}, see \cite{Porta15a}.

Similarly, on the algebraic side, there is a category $\cat Top(\cat T_{\acute et})$ of $\cat T_{\acute et}$-structured topoi, and this category contains a full subcategory of geometric derived stacks equivalent to DM stacks in Toën-Vezzosi's framework.

There is a forgetful functor $(-)^{alg}: \cat Top(\cat T_{an}) \to \cat Top(\cat T_{\acute et})$ which corresponds to forgetting extra structure and considering just the algebraic object $(\cat X, \cat O_{\cat X}^{alg})$.

Then the the following hold: 

\begin{itemize}
\item There is an analytification functor $(-)^{an}$ from $\cat T_{\acute et}$-structured topoi to $\cat T_{an}$-structured topoi, which is right adjoint to the forgetful functor $(-)^{alg}$. This is an adjunction of functors of $\oo$-categories.
See Theorem 2.1.1 in \cite{Lurie11e}.
\item The analytification functor sends DM stacks locally almost of finite presentation to derived analytic spaces, see Remark 12.26 in \cite{Lurie11d}.
\item The analytification functor can be extended to derived Artin stacks, see \cite{Porta15} and also Section 3.1 of \cite{HolsteinF}. 
It follows from Lemma 2.35 of \cite{Porta14} applied to the derived context that analytification sends geometric derived algebraic stacks locally almost of finite presentation to geometric derived analytic stacks. 
\item Analytification restricts to the usual analytification on the subcategory of underived schemes, see Lemma 4.4 in \cite{Porta15}.
\item The natural comparison map $h: (X^{an})^{alg} \to X$ is flat, see Corollary 5.15 in \cite{Porta15}. 
\item Analytification commutes with the truncations $\pi^{0}$ and $\pi_{0}$, see Section 9.2 in \cite{Porta15} and Lemma 2.20 of \cite{Porta14}. 
\end{itemize} 

It is worth adding some explanation about analytification of Artin stacks. The analytification functor as defined by Lurie only applies to DM stacks as it depends on the notion of a structured topos. It can be turned into a functor on Artin stacks as a left Kan extension, see \cite{Porta15}. But as a left Kan extension the new functor does not have an automatic left adjoint. (This was pointed out to the authors by Mauro Porta.)

\begin{rk} The analytification of Artin stacks is only given by a left Kan extension when we consider derived stacks on simplicial algebras locally almost of finite presentation. This suffices for our purposes, but the theory extends to arbitrary derived stacks, see Section 3.1 of \cite{HolsteinF}.
\end{rk}
Many moduli stacks naturally occur as Artin stacks. This is true in particular for the derived period domain we study here. 

In joint work with Mauro Porta \cite{HolsteinF}
the second author examines mapping spaces into analytified Artin stacks. Despite the lack of a left adjoint there are algebraic descriptions of these mapping spaces. 

However, for the purposes of this paper we will sidestep this issue by constructing maps into analytified mapping spaces directly using simplicial presentations. 
 \margin{2.5}

Next we need to recall some characterizations of sheaves on derived analytic stacks.

Definition 6.3 in \cite{Porta15} defines a category of coherent sheaves on a derived analytic space $X$: First one takes the $\oo$-topos $\cat X$ associated to $X$ and observes there is a category of $\cat O_{\cat X}^{alg}$-modules. The category of \emph{coherent sheaves} is the subcategory of modules $\cat F$ whose cohomology sheaves are locally on $\cat X$ coherent sheaves of $\pi_{0}(\cat O_{\cat X}^{alg})$-modules. This is entirely analogous to the algebraic case.

Let us unravel this definition. First we make the definition of a derived analytic space a little more precise. Recall the $\oo$-topos of a topological space $T$ can be thought of as the $(\oo,1)$-category of sheaves of spaces on $T$. A model for its hypercompletion is given by the model category of simplicial presheaves on $T$ (cf. Remark 6.5.0.1 in \cite{Lurie11a}). But a derived analytic space is always hypercomplete, see Lemma 3.1 of \cite{Porta15}. So locally we are just working with the local projective model category structure on simplicial presheaves 
on a topological space. The space is then equipped with a sheaf of simplicial rings $\cat O_{X}^{alg}$. (Plus the extra structure that does not affect these definitions).

Many of the derived analytic spaces in this paper will actually have an underlying topological space, so this description applies globally. (This is also true in the algebraic setting for derived schemes.) There will be a brief discussion about sheaves on stacks in Section \ref{sect-alggrass}.

In general, $\cat O_{\cat X}$-modules are $\cat O_{\cat X}$-module objects in sheaves on an $\oo$-topos $\cat X$, cf. Section 2.1 of \cite{Lurie11d}. 
As we are working over $\C$ we can model this by the model category of sheaves of chain complexes over $N(\cat O_{X}^{alg})$ (using the normalization functor $N$).

There are cohomology groups defined in this setting, which just come down to the usual cohomology groups, see Definition 7.2.2.14 in \cite{Lurie11a}. There is a $t$-structure on $Coh(X)$ and the heart satisfies $Coh^{\heartsuit}(X) \simeq Coh^{\heartsuit}(\pi^{0}(X))$. 

\begin{rk}
This definition of coherent sheaves may look naive, but Porta in \cite{Porta15a}
shows that $\cat O^{alg}_{X}$-modules
are equivalent to the more natural category $Sp(Str^{loc}_{\cat T_{an}}(X)_{/\cat O_{X}})$. 
\end{rk}

The cohomology groups can also be used to define bounded below complexes. Note however that $\cat O_{X}^{alg}$ is not necessarily bounded below. 
As in the underived case there is an analytification functor on sheaves, provided by pullback along the natural map $h_{X}: (X^{an})^{alg} \to X$, which sends bounded below coherent sheaves to bounded below coherent sheaves, see \cite{Porta15}.

\section{The derived period domain}\label{sect-perdomain}

\subsection{The classical period domain}\label{sect-classicaldomain}
In this section we construct the target of the period map, the (polarized) derived period domain . The derived period domain is a derived analytic stack enhancing the usual period domain. It  will classify filtrations of a complex $V$ equipped with a bilinear form $Q$ which satisfy the Hodge-Riemann bilinear relations.

For background we will recall the construction of the period domain as a subspace of a flag variety in the underived case. (Here the flag variety is the moduli of (partial) flags with prescribed dimensions in a vector space $V$. This variety may be described as the quotient $SL(V)/P$ for some parabolic subgroup $P$.) 

Assume we are given a smooth projective family of varieties $X \to S$ with fibre $X_{s}$. The polarized period domain is the moduli of Hodge filtrations on $H^{k}(X_{s}, \C)$, it is a subspace of the flag variety of $H^{k}$.
To be precise we are given a vector space $H^{k} = H^{k}(X_{s}, \set \Z)\otimes_{\Z} \C$ with an integral structure and a bilinear form  $Q_{k}(\al,\beta) = \langle L^{n-k}\al,\beta\rangle$ where $L$ is the Lefschetz operator and $\langle,\rangle$ is the intersection form $\langle\alpha,\beta\rangle = \int \alpha \cup \beta$. Note $Q_{k}$ is bilinear, symmetric in even degrees and anti-symmetric in odd degrees. 
These vector spaces are identified for all $s$ by Ehresmann's theorem as they are diffeomorphism invariants. 
\begin{rk} $Q_{k}$ should not be confused with the hermitian form $Q_{k}(\alpha, \bar \beta)$.
\end{rk}

Now the polarized period domain is an open subset of a closed subset of the flag variety, 
 given by the conditions below, which for ease of reference we call the Hodge-Riemann bilinear relations. (This is a slight abuse of language, more correctly the term is only applied to the first and third condition.)

\begin{enumerate}
\item $F^{p}H^{k} = (F^{k-p+1}H^{k})^{\perp}$ with respect to $Q_{k}$.
\item $H^{k}=F^{p}H^{k}\oplus \overline{F^{k-p+1}H^{k}}$.
\item On $H^{p,q}_{prim} = F^{p}H^{k}_{prim}\cap \overline{F^{q}H^{k}_{prim}}$ we have $(-1)^{\frac{k(k-1)}2}i^{p-q} Q_{k}(\al,\bar \al)> 0$. Here $q = k-p$ and $H_{prim}$ is the primitive cohomology.
\end{enumerate}
The signature of $Q_{k}$ on the non-primitive parts of cohomology is easily worked out.
\begin{rk}
When dealing with all cohomology groups at the same it is convenient to consider the bilinear pairing $Q(\alpha,\beta) = \int \alpha \cup \beta$ instead and we will do so in future sections. The first condition become $Q(F^{p}, F^{n-p+1})=0$. (The second and third condition will still be treated level-wise.) 
\end{rk}

The first condition describes an algebraic subset in the flag variety.
We call the closed subset given by condition (1) only the closure of the  period domain. 

The other two conditions are (analytic) open on the filtrations satisfying the first condition.

\begin{rk}
One can view the period domain in two ways: As an open subspace of a closed subspace of a flag, as above, or as a homogeneous space for the group of symmetries, roughly $\Aut(H_{\set R}, Q)/\Aut_{F}(H_{\set R}, Q) \subset \Aut(H, Q)/\Aut_{F}(H,Q)$.   Note that this is a quotient of real groups, which is then shown to be complex as an open subspace of a quotient of complex groups. 

This approach seems harder to imitate in the derived setting, which is why our starting point is the flag variety.
\end{rk}

Note that the tangent space of the period domain is given by $\End/F^{0}(\End)$, where we write $\End$ for the Lie algebra of endomorphisms of $H^{k}$ compatible with the pairing, filtered in the usual way, i.e. $F^{0}(\End)$ consists of filtered endomorphisms. See Section 1 of \cite{Cattani78}. 
The image of the period map has tangent space contained in the \emph{horizontal tangent space}, given by $F^{-1}(\End)/F^{0}(\End)$, as a consequence of Griffiths transversality.

To construct the period map for a family $X \to S$ one has to take the quotient of the period domain by the monodromy group $\pi_{1}(S)$.
Note that $\pi_{1}(S)$ is a subgroup of $\Aut(H_{\Z},Q)$. It acts on integral cohomology, hence the $\Z$-coefficients, and it preserves the bilinear form.  

The fundamental group acts properly discontinuously on the period domain.

\subsection{The derived flag variety}\label{sect-alggrass}

We will now start constructing the derived period domain. 
We will build it step by step starting from moduli stacks of perfect complexes and filtered complexes.

For simplicity we will define all our stacks on the model category of non-positively graded commutative dg-algebras. We write $A$ for an arbitrary commutative dg-algebra in non-positive degrees. We will consider its model category of dg-modules with the projective model structure (all objects are fibrant). 

Before we start with the constructions let us first recall the following useful fact which we will use repeatedly: 
\begin{lemma}\label{lemma-open}
Suppose we are given an (analytic or algebraic) geometric derived stack $\mathcal X$. Then for any open substack $U$ of $\pi^{0}(\mathcal X)$ there is a unique geometric derived stack $\mathcal U$ with $\pi^{0}(\mathcal U) = U$ that is an open substack of $\mathcal X$.
\end{lemma}
\begin{proof}
See Proposition 2.1 in \cite{Schurg11} for the algebraic case. The analytic case follows from Lemma 7.18 in \cite{Porta16}.

The key point is that $\pi^{0}$ induces an equivalence of sites between $X_{\textrm{\'et}}$ and $(\pi^{0}X)_{\textrm{\'et}}$, i.e.\ between \'etale maps to $X$ and to $\pi^{0}(X)$. 
Open immersions are just \'etale maps which are monic.
\end{proof}

Hence open conditions can be dealt with easily by imposing open conditions on the underlying underived space. 

Now we recall our two main building blocks, the stacks of perfect complexes and of filtered complexes, which are locally geometric derived stacks locally almost of finite presentation.  

We begin with a description of quasi-coherent sheaves on derived stacks. Generally speaking, when we have a description of the $\oo$-category $\cat {QC}oh(U)$ of complexes of quasi-coherent sheaves for any derived affine scheme $U$, the most natural global definition is that $\cat {QC}oh(X) = \holim_{i} \cat {QC}oh(U_{i})$ where $U_{i}$ is a simplicial derived affine scheme whose realization is the derived stack $X$. As $\cat {QC}oh$  is a stack (see for example \cite{Toen05a}) this is well-defined. 

It will be useful later on to have a very explicit model in which sheaves on a stack can be modelled by a (subcategory of a) model category. This can for example be achieved as follows.
Given a cover $U_{i}$ as above consider $QCoh(U_{i})$ to be the dg model category of complexes of quasi-coherent sheaves on $U_{i}$, such that its subcategory of fibrant cofibrant objects $QCoh^{cf}$ is a model for $\cat{QC}oh$. Any $U_{i} \to U_{j}$ induces a Quillen adjunction. These data form a \emph{left Quillen presheaf} $i \mapsto QCoh(U_{i})$, see Appendix B of \cite{Toen05a}.

We define a category such that an object $G \in QCoh(S)$ is given by the following data: For every $i$ there is $G_{i} \in QCoh(U_{i})$ and for every map $\sigma_{ij}: U_{i} \to U_{j}$ there is a comparison map $\phi_{ij}: G_{i} \to \sigma_{ij}^{*}G_{j}$. Morphisms are given by strictly compatibly families of maps.  This is the category of \emph{presections}  with the injective model structure where cofibrations and weak equivalences are defined level-wise.
Then we consider those objects $G_{\bullet}$ such that all $\phi_{ij}$ are weak equivalences. These are called the \emph{homotopy cartesian sections} and we consider the category of homotopy cartesian sections which are moreover fibrant and cofibrant. This dg category is a model for $\cat{QC}oh(S)$. 

This result is called \emph{strictification}, it seems to be part of the folklore. An early reference is \cite{Simpson01}, see also Appendix B of \cite{Toen05a}. A detailed and accessible proof under some finiteness assumptions can be found in \cite{Spitzweck10}. These finiteness assumptions are satisfied if $S$ is of finite presentation, cf.\ the proof of Theorem \ref{mainA}.

As pointed out by an anonymous referee, a general strictification result can be obtained from the results in \cite{Lurie11a}. We give a sketch of their argument. The key ideas are as follows: The homotopy limit of the diagram $i \mapsto QCoh(U_{i})$ can be computed as a quasi-categorical limit.
The $\oo$-functor $R\Spec(A) \mapsto A\mods$ from  derived affine schemes over $S$ considered as an $\oo$-category to the $\oo$-category of chain complexes can be turned into a cartesian fibration $p: \mathbf{Mod} \to (dAff/S)^{op}$ of $\oo$-categories by the yoga of straightening/unstraightening. Then unravelling definitions the total space $\mathbf{Mod}$ of the fibration has a presentation coming from a model category.
The limit of this diagram may be written as a category of cartesian sections in the $\oo$-categorical sense by Corollary 3.3.3.2 of \cite{Lurie11a}. 
Now the $\oo$-category of all sections can be strictified using Proposition 4.2.4.4 of \cite{Lurie11a}, and cartesian sections are a left Bousfield localization of the $\oo$-category of all sections.

\begin{defn}
Let $Perf$, the moduli stack of perfect complexes, be defined as follows. 

For any derived stack $S$ consider a simplicial derived affine hypercover $U_{\bullet}$ and let $\cat Perf(S)$ be the category of fibrant cofibrant homotopy cartesian sections $G_{\bullet}$ of the left Quillen presheaf $QCoh$ such that each $G_{i}$ is perfect as a $\cat O(U_{i})$-module. 
We say that an $A$-module $M$ over a dg-algebra $A$ is perfect if it is compact in the derived category of $A$. This is equivalent to saying $M$ is pseudo-coherent (also called almost perfect) and of finite Tor dimension. 
The last condition is equivalent (see \cite{Toen07}) to saying $M\otimes^L_{A}H^{0}(A)$ is a perfect complex of $H^{0}(A)$-modules, i.e.\ is
quasi-isomorphic to a bounded complex of finite projective $H^{0}(A)$-modules.

In particular $\cat Perf$ sends $\Spec(A)$ to 
the dg-category of cofibrant $A$-modules quasi-isomorphic to strictly perfect dg modules over $A$. 
Then $Perf(S) \coloneqq  N_{W}(\cat Perf(S))$. (In the literature this is sometimes denoted $\set RPerf$.) Recall that when taking $N_{W}$ we restrict the morphisms to allow only the ones that become isomorphisms in the homotopy category. 

For a proof that this is a locally geometric derived Artin stack locally almost of finite presentation, see \cite{diNatale14a}. Note that this stack is equivalent to the construction of $\cat M_{\id}$ in \cite{Toen07}. 
\end{defn}

\begin{defn}
Let $Filt_{n}$ be the stack of filtered perfect complexes of filtration length $n+1$. 

Explicitly, we let $\cat Filt^{mod}_{n}(A)$ be the 
 dg model category
of sequences $F^{n} \to \dots \to F^{1}\to F^{0}$ of injections of $A$-modules as defined in \cite{diNatale14a}. In particular, weak equivalences are given by filtered quasi-isomorphisms, i.e.\ morphisms inducing quasi-isomorphisms on the associated graded modules.
$\cat Filt_{n}(A)$ is the dg subcategory of fibrant cofibrant sequences $F^{\bullet}$ such that all $F^{i}$ are perfect. 

\begin{rk}
For later examples it is reassuring to note that it follows from Propositions 1.19 and 1.21 in \cite{diNatale14a} that any bounded filtered complexes over $\C$ is fibrant and cofibrant.
\end{rk} 
$\Hom_{\cat Filt_{n}(A)}(F^{*},G^{*})$ consists of compatible maps $F^{i} \to G^{i}$. These maps are determined on $F^{0}$, and 
we will sometimes write the derived hom complex as $\uHom_{F}(F^{0},G^{0})$ if the filtrations are clear. 

We note that the associated simplicial category may also be defined by Dwyer-Kan localising the category of (not necessarily fibrant and cofibrant) filtrations at filtered quasi-isomorphisms.

We will sometimes represent points of $\cat Filt_{n}$ by filtrations which are not fibrant or cofibrant, but will always consider the derived hom space $\uHom_{F}$. 

 Analogously to $\cat Perf(S)$ we can define the dg category $\cat Filt_{n}(S)$ as consisting of certain fibrant cofibrant homotopy cartesian sections of the left Quillen presheaf $U_{i} \mapsto \cat Filt^{mod}_{n}(\cat O(U_{i}))$.
 
 We then let $Filt_{n}$ be $N_{W}(\cat Filt_{n})$. That this is a locally geometric derived stack locally almost of finite presentation can be deduced from Theorem 2.33 and Remark 2.34 in \cite{diNatale14a}, where it is shown that $\cup_{n} Filt_{n}$ and $Filt_{1}$ are locally geometry. (Note that our subscript $n$ has a different meaning from the superscript $n$ in $\cat Filt^{n}$ in loc.\ cit.) 
 \end{defn}
 \begin{rk}
We note that the category $\cat Filt_{n}(A)$ as constructed here and in \cite{diNatale14a} is not stable.
The stable version is given by strings of $n$ composable arrows in $Perf(A)$ without the injectivity condition. 
These two constructions give rather different categories, but it follows from the proof of Theorem 4.4 of \cite{Gwilliam16} that the associated homotopy categories and mapping spaces agree, and thus the two constructions give rise to the same stack.
\end{rk}

 The next stack also appears in \cite{diNatale14a}, but we will revisit the construction to give an explicit description of the mapping spaces.
 \begin{lemma}
For  a complex $V$ consider $\cat Flag'_{n}(V)(A) = \holim(\cat Filt_{n} (A)\rightrightarrows \cat Perf(A))$ where the two maps are induced by the forgetful functor $\pi: F^{*} \mapsto F^{0}$ and the constant functor $V: F^{*} \mapsto V$. This is the dg-category of pairs $(F^{*}, w_{F})$ where $F^{*}$ is an object of $\cat Filt_{n}(A)$ and $w_{F}: F^{0} \to V\otimes A$ is a homotopy equivalence in $\cat Perf(A)$. 

The morphism complex is given by 
\[
\uHom((F^{*},w_{F}), (G^{*},w_{G})) = (\uHom_{F}(F^{0},G^{0}) \oplus \uHom(F^{0},A[1]), \Delta)\]
where $\Delta: (f, h) \mapsto (-df, dh + w_{G}\circ f - w_{F}\circ \delta)$ where $\delta$ is the identity if $f$ has degree $0$ and $0$ otherwise. 
\end{lemma}
\begin{proof}
To show that the characterization of the homotopy limit is correct we replace $\cat Filt_{n}$ by $\cat Filt_{n}^{\sim}$, which has objects consisting of triples $(F^{*}, M_{F}, m_{F})$ where $F^{*}$ is an object in $\cat Filt_{n}$ and $m_{F}: F^{0} \to M_{F}$ is a homotopy equivalence in $\cat Perf$. Morphism complexes are given by $(\uHom_{\cat Filt}(F^{0}, G^{0}) \oplus \uHom(M_{F},M_{G}) \oplus \uHom(F^{0}, M_{G})[1], \Delta)$ where $\Delta: (f,g,h) \mapsto (-df,dg,dh  + m_{G}\circ f- g\circ m_{F})$.

The natural inclusion is a quasi-equivalence and the map $\cat Filt^{\sim}_{n} \to \cat Perf$ is a fibration. This follows from the standard arguments used in the construction of the path space in dg-categories, see Section 3 of \cite{Tabuada10a}.

As dg-categories (with the Dwyer-Kan model structure) are right proper we can compute the homotopy limit as the limit of $\cat Filt_{n}^{~} \to \cat Perf \leftarrow *$, which is the category described in the Lemma.
\end{proof}

\begin{lemma}\label{lemma-dflag}
The derived flag variety $DFlag_{n}(V)$ is defined as $N_{W}$ of the substack  $\cat Flag_{n}(V)'$ that is given by filtrations such that the maps $H^{k}(F^{i-1}\otimes^L_{A}
H^{0}(A)) \to  H^{k}(F^{i}\otimes^L_{A}H^{0}(A))$ are injections of flat $H^{0}(A)$-modules. It is geometric.
\end{lemma}
\begin{proof}
We first note that $N_{W}(\cat Flag'_{n}(V))$ is the homotopy limit of $Filt_{n} \rightrightarrows Perf$ as $N_{W}$ commutes with homotopy fibre products of simplicial categories whose homotopy categories are groupoids.
As a homotopy pullback of locally geometric stacks it is locally geometric.
 
To see the substack is locally geometric, use Proposition 2.44 in \cite{diNatale14a}.

We now want to refine this result and show that $DFlag_{n}(V)$ is geometric. We let $k-1$ be the amplitude of the complex $V$ and will show that $DFlag_{n}(V)$ is $(k+2)$-geometric.
We introduce the $k$-geometric stack $Filt_{n}^{k}(A)$ which classifies filtrations $F^{*}$ which satisfy that $\Ext^{i}_{\cat Filt_{n}(H^{0}(A))}(F^{*}\otimes^L_{A}H^{0}(A), F^{*}\otimes^L_{A}H^{0}(A)) = 0$ for $i < -k$.
This is a derived geometric $k$-stack by Theorem 2.33 of \cite{diNatale14a}, and thus is $(k+2)$-geometric. 
Now we check that all the points in $DFlag_{n}(V)$ map to $Filt^{k}_{n}(A) \subset Filt_{n}(A)$. We need to compute Ext-groups for filtered complexes. 
We do this by considering the Rees construction $\oplus F^{i}t^{i}$, cf. \cite{diNatale14a}. Then $\Ext^{*}_{\cat Filt(H^{0}(A))}(F^{*},F^{*}) = \Ext_{H^{0}(A)[t]}^{*}(\oplus F^{i}t^{i},\oplus F^{i}t^{i})^{\set G_{m}}$. As $\set G_{m}$ is reductive we can ignore it when checking vanishing of Ext-groups. All $F^{i}\otimes^L_{A}H^{0}(A)$ have amplitude bounded by the amplitude of $V$. Moreover, they are projective, thus there is a two-term resolution of $\oplus F^{i}t^{i}$ by projective $R[t]$-modules (see e.g.\ Theorem 4.3.7 in \cite{Weibel95}). Thus the Ext groups vanish in degrees less than $k$.

It follows that we can construct $DFlag_{n}(V)$ as an open substack of the homotopy fibre of $Filt_{n}^{k} \to Perf^{k}$, where $Perf^{k}$ is defined analogously to $Filt_{n}^{k}$. Thus $DFlag_{n}(V)$ is $(k+2)$-geometric. 
\end{proof}

Moreover one can show that the derived Grassmannian $DFlag_{1}(V)$ is an enhancement of the usual Grassmannian, i.e.\ $\pi_{0}\pi^{0}D Flag_{1}(V) = \prod_{k_{i}, i}Gr(k_{i}, H^{i}(V))$, see Theorem 2.42 in \cite{diNatale14a}, and similar for the derived flag varieties.

The tangent complex of $DFlag_{n}(V)$ at $(F^{*}, w_{F})$ is computed by the homotopy limit of the tangent complexes:
\[T_{F^{*}, w_{F}}DFlag(V)=\operatorname{cone}(({\chi,0} ):\holim(\uEnd_{F}(F^{0}) \rightrightarrows \uEnd(F^{0}))[1]
\]
  Here $\uEnd_{F}(V)$ is the derived endomorphism space in the model category of filtrations, $\chi$ is the inclusion and $0$ is the constant zero map. 
For this and similar cones we will write $\uEnd(F^{0})/\uEnd_{F}(F^{0})$

\subsection{The derived period domain I: Algebraically}\label{sect-algdomain}

As the next step towards constructing the derived period domain we
construct in this section a geometric derived algebraic stack $D$ that classifies filtrations with a bilinear form which satisfy only the Hodge-Riemann orthogonality condition.
The derived period domain $U$ will then be the open substack of $D^{an}$ determined by the second and third Hodge-Riemann condition (positivity). 

We fix a non-negative integer $n$ and a complex $V = V_{\set Q} \otimes \C$ 
concentrated in degrees $0$ to $2n$ with a $2n$-shifted non-degenerate symmetric bilinear form. Clearly this map can  be extended by multiplication to a bilinear map on $V\otimes A$.

\begin{defn}
Given an $A$-module $W$ we denote by $\Sym^{2}W$ the derived coinvariants of the natural $\Z/2\Z$-action on $V\otimes V$ and we call a map  $Q: \Sym^{2}W \to A[k]$ for some integer $k$ a \emph{shifted bilinear form}. $Q$ is called \emph{non-degenerate} if the associated map $q: W \to W^{\vee}[2n]$ is a weak equivalence.  ]
\end{defn}
We want to consider the moduli stack $D_n(V,Q)$ of filtrations $F^{*}$ on the complex $V$ such that $F^{i}$ and $F^{n-i+1}$ are orthogonal with respect to $Q$. 

\begin{thm}\label{thm-periodstack}
Let $(V,Q)$ be as above.
Then there is a geometric derived stack $D = D_n({V,Q})$ which classifies decreasing filtrations of $V$ of length $n+1$ that satisfy the Hodge-Riemann orthogonality relation with respect to $Q$. $D$ enhances the product of closures of classical period domains.

Here $D = N_{W}(\cat D_n({V,Q}))$ where $\cat D_n(V,Q)(A)$ is a certain simplicial category whose homotopy category is given as follows: Objects are given by triples $(F^{*},  w_{F}, Q_{F})$ where
\begin{itemize}
\item $F^{*}$ is a filtration of perfect $A$-modules of length $n+1$,
\item $Q_{F}$ is a non-degenerate $2n$-shifted bilinear form on $F^{0}$ that vanishes on $F^{i}\otimes^{L}_{A} F^{n+1-i}$ for all $i$,
\item $w_{F}: (F^{0},Q_{F}) \simeq (V\otimes A, Q)$ is an isomorphism in the homotopy category of perfect complexes with shifted bilinear form, 
\end{itemize}
such that $F^{*}$ gives filtrations on cohomology after tensoring over $A$ with $H^{0}$ and all $H^{j}(F^{i}\otimes^L_{A}H^{0}(A))$ are flat. 

Morphisms from $(F^{*}, w_{F}, Q_{F})$ to $(G^{*}, w_{G}, Q_{G})$ are given by families of morphisms $F^{i}\to G^{i}$ in the homotopy category of $\cat Filt_{n}(A)$ that are compatible with $w_{\bullet}$ and $Q_{\bullet}$.  
\end{thm}

\begin{rk} In the case of interest to us $V$ will have a real structure and $Q$ will be real. We note that just as in the underived case the real structure only becomes relevant when considering the positivity condition for the period domain.
\end{rk}

We will construct $\cat D_{n}(V,Q)$ by adding the data of a quadratic form to $\cat Flag_{n}(V)$. As a warm-up and for use in later sections 
we first prove the following result about a stack constructed by Vezzosi in \cite{Vezzosi13}:

\begin{lemma}\label{lemma-perfbilinear}
There is a locally geometric derived stack of perfect complexes with a $2n$-shifted non-degenerate bilinear form, denoted $QPerf$. 

Moreover $QPerf(A)=N_{W}(\cat{QP}erf(A))$ for a simplicial category $\cat {QP}erf(A)$ whose objects 
are given by perfect complexes $W$ over $A$ with a bilinear form $Q: \Sym^{2}W \to A[2n]$ that is non-degenerate.
The morphism space $\Map_{\cat{QP}erf(A)}((W,Q), (W',Q'))$ is given by the homotopy fibre of the map $f \mapsto Q'\circ\Sym(f): \Map_{\cat Perf(A)}(W,W') \to \Map(\Sym^{2}W, A[2n])$ over $Q$.
\end{lemma}

\begin{proof}
This is the derived 
stack that is constructed in Section 3 of \cite{Vezzosi13} and denoted by $QPerf^{nd}(2n)$ there. We recall the construction and prove the properties we need.

We use simplicial categories as our model for $(\oo,1)$-categories. First we define $QPerf'$ using the sheaf of simplicial categories $\cat {QP}erf'$ defined via the following homotopy pullback diagram.
\[
\bfig
\Square[\cat {QP}erf'(A)`\cat Perf(A)^{\Delta^{1}}`\cat Perf(A)`\cat Perf(A)\times \cat Perf(A);``ev_{0}\times ev_{1}`\Sym^{2}_{A} \times A[2n{]}]
\efig
\] 

Here $(-)^{\Delta^{1}}$ denotes the functor category from the two object category with one morphism, $* \to *$. Then $ev_{0}\times ev_{1}$ denotes the product of evaluation maps. $\Sym^{2}_{A}$ is the functor $V \mapsto k \otimes^L_{k[\Sigma_{2}]} (V \otimes^L_{N(A)} V)$, where the $\Sigma_{2}$-action is given by swapping factors.

To compute the homotopy limit we need to replace $\cat Perf^{\Delta^{1}} \to \cat Perf\times \cat Perf$ by a fibration. We consider the dg-category $\cat Perf^{\Delta^{1},\sim}$ which replaces morphisms in $\cat Perf^{\Delta^{1}}$, which by definition are commuting squares, by homotopy commutative squares, i.e.\ in $\cat Perf^{\Delta^{1},\sim}$ we have $\uHom(f:K\to L,g:M\to N) = (\uHom(K,M)\oplus \uHom(L,N)\oplus \uHom(K,N)[1]), \Delta$ where $\Delta: (x,y,h) \mapsto (dx,dy, dh+y\circ f - g\circ x)$. As in $\cat Perf$ we only consider morphisms which become invertible in the homotopy category.
Then we apply the Dold-Kan functor which preserves fibrations.

Unravelling definitions the limit is the category of perfect complexes with a bilinear form.  The objects are given by objects $W \in \cat Perf(A)$ and $Q: \Sym^{2}W\to A[2n]$ in $\cat Perf(A)^{\Delta^{1}}$. 
Morphisms from $(W,Q_{W})$ to $(U, Q_{U})$ are given by the  fibre of $f \mapsto Q'\circ\Sym^{2}(f): \Map(W, U) \to \Map(\Sym^{2}W, A[2n])$ over $Q$.

We apply $N_W$ to $\cat{QP}erf'$ and call the resulting derived stack $QPerf'$. We know that $N_{W}$ respects this homotopy limit. Thus by Lemma \ref{lemma-perfdelta} below $QPerf'$ is a homotopy limit of locally geometric stacks and thus locally geometric.

The condition that $Q$ is non-degenerate is easily seen to be open: It is enough to check on open subsets of the underlying underived stack, and being quasi-isomorphic is an open condition on complexes. 
Thus the substack of non-degenerate bilinear forms $QPerf(2n) \subset QPerf'(2n)$ is also locally geometric. 

It is clear from these definitions that $\pi_{0}\pi^{0}(QPerf(2n))(A)$ consists of isomorphism classes of perfect complexes over $H^{0}(A)$ with a non-degenerate quadratic form. 
\end{proof}
\begin{rk}
The reason we have to work with simplicial categories is of course that $f \mapsto \Sym^{2}(f)$ is not a linear map. 
\end{rk}

\begin{lemma}\label{lemma-perfdelta}
$N_{W}(\cat Perf^{\Delta^{1}})$ is a locally geometric stack.
\end{lemma}
\begin{proof}
We can see this by identifying our stack with the linear stack associated to the perfect complex $E = \cat Hom(p_{1}^{*}\cat U, p^{*}_{2}\cat U)^{*}$ on the locally geometric stack $\cat Perf \times \cat Perf$. Here $\cat U$ is the universal complex on $\cat Perf$.

To\"en defines linear stacks in Section 3.3 of \cite{Toen14}.
For a derived Artin stack $X$ with quasi-coherent complex $E$ we define the stack $V_{E}$ as a contravariant functor on derived schemes over $X$ by associating $\set V_E(u) = \Map_{Lqcoh(S)}(u^{*}E, \cat O_{S})$ to $u: S \to X$. 
Now $\set V_E$ is locally geometric 
whenever $E$ is perfect, the proof follows from Sublemma 3.9 in \cite{Toen07}. 

To show this agrees with our construction above, we consider the definition of $Perf^{\Delta^{1}}$ locally. Over $(V,W): S \to Perf\times Perf$ we have:
\[
\Map_{Perf \times Perf}((V,W), Perf^{\Delta^{1}}) 
= N_{W}\left(\cat Perf^{I}(S)) \times_{\cat Perf(S) \times \cat Perf(S)} \{(V,W)\}\right)
\]

On the other hand $\set V_{E}((V,W)) = \Map_{S}(\cat Hom(V,W)^{*}, \cat O_{S})$. Both of these compute the mapping space from $V$ to $W$ in $\cat Perf(S)$. (The characterization in terms of the fibre product can for example be found in Sections 1.2 and 2.2 of \cite{Lurie11a}.)
\end{proof}

\begin{rk}
There is another explicit description of $QPerf$: as an open substack of the linear stack associated to the perfect complex $\cat Hom(\Sym^{2}\cat U, \cat O[2n])$ on $Perf$, following the same arguments as Lemma \ref{lemma-perfdelta}.

To show this agrees with our construction above, consider the definition of $QPerf$ locally. (Looking at the simplicial set, not the dg-category.) Over $V: S \to Perf$ we have:
\begin{eqnarray*}
\Map_{Perf}(V, QPerf) 
&=& \Map(S, QPerf) \times_{\Map(S, Perf)} \{V\} \\
&=& \left(Perf(S) \times_{(Perf\times Perf)(S)} Perf^{\Delta^{1}}(S))\right) \times_{Perf(S))}\{V\} \\
&=&  \left\{ (\Sym^{2}V, \cat O) \right\}
\times_{Perf(S)\times Perf(S)}Perf(S)^{\Delta^{1}}
\end{eqnarray*}
And the last line is precisely the mapping space from $\Sym^{2}V$ to $\cat O$ in $Perf(S)$, i.e.\ it is $\set V(\Sym^{2}(\cat U))(V)$.

Applying shift and dual does not affect the argument, hence we are done.
\end{rk}
We can use the same technique as in Lemma \ref{lemma-perfbilinear} to consider filtered complexes with a bilinear form:
\begin{lemma}\label{lemma-filtbilinear}
There is a locally geometric derived stack $QFilt_{n}$ of perfect complexes with an $(n+1)$-term filtration and a non-degenerate bilinear form compatible with the filtration.

To be precise $QFilt_{n} = N_{W}(\cat{QF}ilt_{n})$ for a simplicial category $\cat QFilt_{n}$ whose objects are filtrations of perfect $A$-modules $F^{n}\to F^{n-1}\to \dots \to F^{0}$ together with a non-degenerate bilinear form $Q: \Sym^{2}F^{0}\to  A[2n]$ such that $Q$ vanishes on the image of $F^{i}\otimes^L_{A} F^{n+1-i}$ for all $i$.

The morphism space $\Map_{\cat{QP}erf(A)}((F^{*},Q_{F}), (G^{*},Q_{G}))$ is given by the fibre of the map $f \mapsto Q_{G}\circ\Sym(f): \Map_{\cat Filt_{n}}(F^{*},G^{*}) \to \Map_{\cat Filt_{1}}(\Sym^{2}F^{*}, A[2n])$ over $Q_{F}$.
\end{lemma}

\begin{proof}
We proceed like in the proof of Lemma \ref{lemma-perfbilinear} and consider the homotopy pullback diagram
\begin{equation}\label{diagram-qfiltpullback}
\bfig
\Square[\cat {QF}ilt'_{n}(A)`\cat Filt_{1}(A)^{\Delta^{1}}`\cat Filt_{n}(A)`\cat Filt_{1}(A)\times \cat Filt_{1}(A);``ev_{0}\times ev_{1}`\cat S \times (0 \to A[2n{]})]
\efig
\end{equation} 

Here the map $\cat S$ sends $F^{n}\to \dots \to F^{0}$ to $\cat S_{F} \to \Sym^{2}F^{0}$ where $\cat S_{F}$ is the image of $\oplus_{i}(F^{i}\otimes^L_{A} F^{n+1-i})$ under $F^{0}\otimes_{A}^{L} F^{0}\to\Sym^{2}F^{0}$. 

We replace the right vertical map by a fibration in the same manner as in Lemma \ref{lemma-perfbilinear}. 

Then we see that the objects of $\cat {QF}ilt'_{n}(A)$ are filtrations with a bilinear form that vanishes on $\cat S_{F}$, i.e.\ satisfies Hodge Riemann orthogonality. Similarly we write down the mapping spaces.

Again we see that $QFilt'_{n} = N_{W}(\cat {QF}ilt'_{n})$ is a homotopy limit of locally geometric stacks and thus locally geometric. We can show $N_{W}(\cat Filt_{1}^{\Delta^{1}})$ is locally geometric by an analogue of \ref{lemma-perfdelta}.

 Again the substack of non-degenerate forms $QFilt_{n}$ is an open substack.
\end{proof}
\begin{rk}
A natural way to understand the orthogonality relation is to consider the map $q: F^{0}\to (F^{0})^{\vee}$ induced by $Q$. The filtration $F^{*}$ induces a filtration $(F^{0}/F^{i})^{\vee}$ and the orthogonality says precisely that $q$ respects this filtration.
\end{rk}

We are now ready to construct the stack $D_{n}(V,Q)$.

\begin{proof}[Proof of Theorem \ref{thm-periodstack}]
To define $\cat D_{n}(V,Q)$ we again repeat the construction of Lemma \ref{lemma-perfbilinear}, this time with $DFlag$ in place of $Perf$.

We recall that the homotopy fibre over $V$ of the projection map $\cat Filt_{k} \to \cat Perf$ is given by $\cat Flag_{k}(V)$. 
While it is possible to compute the homotopy limit directly, 
the homotopy category of the homotopy fibre product can be computed more efficiently using (the proof of) Corollary 1.12 in \cite{Pridham10}.

Assume we are given a diagram $\cat A \stackrel F \to \cat B \stackrel G\leftarrow \cat C$ of simplicial categories such that all their homotopy categories $\pi_{0}(\cat A)$ etc.\ are groupoids. Then the cited result says we can compute $\pi_{0}(\cat A\times^{h}_{\cat B}\cat C)$ as $\pi_{0}(\cat A)\times^{(2)}_{\pi_{0}(\cat B)} \pi_{0}(\cat C)$. Here the 2-fibre product of categories is defined to have as objects triples $(a\in \cat A, c\in \cat C, \omega : F(a)\cong G(c))$ and morphisms from $(a,c,\omega)$ to $(a',c',\omega')$ are just given by morphisms $f: a\to a'$ and $g: c\to c'$ such that $G(g)\circ \omega = \omega' \circ F(f)$.

We consider first the homotopy fibre of the map of simplicial categories $\cat Filt_{1}^{\Delta^{1}} \to \cat Perf^{\Delta^{1}}$ 
over $Q: \Sym^{2}V\to A[2n]$. We call it $\cat Flag_{1}^{\Delta^{1}}(Q)$, abusing notation, and we note it is  geometric by imitating the proof of Lemma \ref{lemma-dflag}. 

By the above, objects of the homotopy category can be written as diagrams
\[
\bfig
\hSquares[M^{0}`M^{1}`\Sym^{2}V`N^{0}`N^{1}`A;`w_{M}`f_{0}`f_{1}`Q``w_{N}]
\efig 
\] 
where the left hand square commutes strictly and the right hand square is a commutative diagram in the homotopy category. 

$\cat D_{n}'(V,Q)$ is defined as the homotopy pullback of the following diagram (suppressing $A$ from the notation): 
\[
\bfig
\Square[\cat {D}'_{n}(V,Q)`\cat {F}lag_{1}(Q)^{\Delta^{1}}`\cat Flag_{n}(V)`\cat Flag_{1}(\Sym^{2}V)\times \cat Flag_{1}(A[2n{]});``ev_{0}\times ev_{1}`\cat S \times (0 \to A[2n{]})]
\efig
\] 

We use the same tool again to compute $\cat D'$ as the 2-fibre product.
We find that its homotopy category has objects given by triples $(F^{*}, w_{F}, Q_{F})$ consisting of a filtration $F^{*}$, an isomorphism $w_{F}: F^{0}\to V\otimes A$ in the homotopy category and a symmetric bilinear form $Q_{F}: \Sym^{2}F^{0} \to A$ compatible with $Q$ that vanishes on $\cat S_{f} \subset \Sym^{2}(F^{0})$.

Morphisms are just homotopy commutative collections of object-wise morphisms.

Then $D' \coloneqq N_{W}(\cat D')$ is geometric as a homotopy limit of geometric stacks. 

We define $D(A)$ to be the subcategory of $D'(A)$ whose objects satisfy that
$H^{j}(F^{i}\otimes^L_{A}H^{0}(A)) \to H^{j}(F^{i-1}\otimes^L_{A}H^{0}(A))$ are injections of flat $H^{0}(A)$-modules. These are open conditions on $D'$, the proof follows as in Proposition 2.41 of \cite{diNatale14a}.

It remains to check that $D$ enhances the period domain, i.e.\ that $\pi_{0}\pi^{0}(D)(A)$ is a product of the usual period domains.

To do this we assume $A$ is an algebra (considered as a constant simplicial algebra). We observe that the objects of $D(A)$ are formal, i.e. any $(F^{*}, w_{F}, Q_{F})$ is equivalent to some filtration on the complex $V\otimes A$ with zero differential. As the complex $F^{i}$ is perfect and the cohomology groups of the $F^{i}$ are flat they are projective
and the complexes are formal. Thus we get a homotopy equivalence from $F^{n} \to \dots \to F^{0} \simeq V\otimes A$ to $H(F^{n}) \to \dots \to H(F^{0}) \cong V \otimes A$, and these are the objects parametrized by the product of closed period domains. 
\end{proof}

\begin{rk}\label{rk-firstdescription}
We note that we can equivalently define $\cat D'(A)$ as the homotopy limit $\holim((\pi,V): \cat{QF}ilt_{n}(A) \rightrightarrows \cat{QP}erf(A))$. 
The arrows are induced by the forgetful functor $\pi: (F^{\bullet}, Q_{F}) \mapsto (F^{0}, Q_{F})$ and the constant functor that sends any filtration to $(V,Q)$. 

As homotopy limits commute this description as a homotopy fibre at $(V,Q)$ of the projection $\cat{QF}ilt \to \cat{QP}erf$ of homotopy limits is the same as the homotopy pullback of homotopy fibres over $V$ and $Q$ that we used above.
\end{rk}

Recall that we gave a description of sheaves on a derived stack $S$ as fibrant cofibrant homotopy cartesian sections of a left Quillen presheaf. 
This description works not just for $\cat Perf$ and $\cat Filt_{n}$ but also $\cat{QF}ilt_{n}$ and indeed $\cat D_{n}(V,Q)$
There is a slight subtlety 
in this model: The presection $U \mapsto \cat O_{S}(U)$ is not fibrant, so where it appears, for example in the definition of shifted bilinear forms $\cat{QF}ilt$, it needs to be replaced fibrantly. Of course this does not affect the homotopy category, so we will not make this explicit if there is no ambiguity.

As an example we state the following 
global version of Theorem \ref{thm-periodstack} to characterize maps from non-affine derived stacks to $D_{n}(V,Q)$. Here by a perfect complex we mean a fibrant cofibrant homotopy cartesian section as above, and we write $P$  for fibrant 
 cofibrant replacement in the model category of presections. 
 
\begin{cor}\label{cor-globalperiodstack}
Let $S$ be a derived stack. Then the objects of $\cat D_{n}(V,Q)(S)$ may be given by triples $(F^{*}, w_{F}, Q_{F})$ where
\begin{itemize}
\item $F^{*}$ is a filtration of length $n+1$ of perfect complexes on $S$
\item $Q_{F}: \Sym^{2}F^{0} \to P\cat O_{S}[2n]$ is a non-degenerate $2n$-shifted bilinear form on $F^{0}$ that vanishes on $F^{i}\otimes_{A}^{L} F^{n+1-i}$ for all $i$,
\item $w_{F}: (F_{0},Q_{F}) \simeq (V\otimes \cat O_{S}, Q)$ is an isomorphism in the homotopy category of complexes with the shifted bilinear forms, 
\end{itemize}
such that $F^{*}$ gives filtrations on cohomology after tensoring over $\cat O_{S}$ with $H^{0}(\cat O_{S})$ and all $H^{j}(F^{i}\otimes^L_{\cat O_{S}}H^{0}(\cat O_{S}))$ are flat. 
\end{cor}

\subsection{The tangent space}\label{sect-tangent}
\begin{propn}\label{propn-tangent}
The tangent complex of $D_n(V,Q)$ at the $\C$-point $(F^{\bullet}, w_{F}, Q_{F})$ is given by $\uEnd_{Q_{F}}(F^{0})/\uEnd_{F, Q_{F}}(F^{0})[1]$. 

Here $\uEnd_{Q_{F}}(F^{0})$ is the complex of endomorphisms of $F^{0}$ that are antisymmetric with respect to $Q_{F}$, i.e. satisfy $Q_{F}(f-,-)+ (-1)^{|f|}Q(-,f-) = 0$. $\uEnd_{F,Q_{F}}(F^{0})$ are those 
endomorphisms that moreover respect the filtration. 
\end{propn}

\begin{proof}
We obtain this result by computing the tangent complexes of $QPerf$ and $QFilt_{n}$ and appealing to the description of $D'_{n}(V,Q)$ as their homotopy fibre, cf.\ Remark \ref{rk-firstdescription}.

As $QPerf = \holim(Perf \to Perf \times Perf \leftarrow Perf^{\Delta^{1}})$ and the tangent space commutes with limits we have the following pullback diagram:

\[\bfig
\Square(0,0)[T_{Q: \Sym^{2}V \to \cat O}QPerf`T_{Q}Perf^{\Delta^{1}}`T_{V}Perf`T_{\Sym^{2}V, \cat O}(Perf \times Perf);```(d\Sym^{2}, 0)]
\efig\]

Let us analyse the ingredients. We recall that all our stacks come from sheaves of dg-categories and then we use Corollary 3.17 from \cite{Toen07}
which tells us that to compute the tangent space of the moduli stack of objects of a triangulated dg-category of finite type it suffices to compute the hom-spaces, i.e. $T_{D}(\cat M) = \uEnd(D)[1]$. 
This applies since we can consider $Perf$ and $Perf^{\Delta^{1}}$ as the moduli stack of objects of a dg-category of finite type. 

Thus it follows that 
\begin{eqnarray*}
T_{V}Perf &=& \uEnd(V)[1] \\
T_{\Sym^{2}V, \cat O}(Perf \times Perf)  &=& \left(\uEnd(\Sym^{2}V)\oplus \uEnd(A)\right)[1]\\
T_{f:A \to B}(Perf^{\Delta^{1}}) &=& \left(\uEnd (A) \oplus \uEnd(B) \oplus \uHom(A,B)[1], \Delta\right) [1] 
\end{eqnarray*}
Here the last complex has differential $\Delta: (x,y,h) \mapsto (dx,dy, dh + gx - yf)$.
This follows from considering hom-spaces in the category $\cat Perf^{\Delta^{1}}$. 

Now to consider the map $d\Sym^{2}$ we have to change to simplicial categories. It is clear that the induced map on derivations sends $f$ to $f\otimes \id + (-1)^{|f|}\id \otimes f$. 

Putting this together we obtain that the tangent space at $(V, Q)$ is the homotopy kernel of the map $e_{Q}: \uEnd(V)\to \uHom(\Sym^{2}V, A[2n])$ defined by $e_{Q}(f)\coloneqq Q(f-,-)+(-1)^{|f|}Q(-,f-)$. We write the kernel as $\uEnd_{Q}(V)$. Over $\C$ the complexes are formal and as $Q$ is non-degenerate $e_{Q}$ is surjective, so these are just endomorphisms that are anti-symmetric with respect to $Q$.

We proceed similarly for $QFilt$ as defined in diagram (\ref{diagram-qfiltpullback}). We use Corollary 3.17 from \cite{Toen07} again to see that the tangent space of $\cat Filt_{n}(A)$ at $F^{\bullet}$ is given by endomorphisms of $F^{0}$ preserving the filtration.
Then $d\cat S$ sends a morphism $f$ respecting the filtration $F$ to $(f\otimes \id + (-1)^{|f|}\id \otimes f)$, which respects the filtration $\cat S_{F} \subset \Sym^{2}F^{0}$.
Thus we find that $T_{F^{*}, Q_{F}}QFilt$ 
consists of filtered endomorphisms of $F^{0}$ that are compatible with $Q_{F}$.

Then $T_{(F^{*}, w_{F}, Q_{F})}D$ is the homotopy equalizer of the natural inclusion and the zero map. (We use the identification of the tangent spaces to $QPerf$ at $(V\otimes A, Q)$ and $(F^{0}, Q_{F})$.) We conclude by noting that $\uEnd_{F,Q}(F^{0})\to \uEnd_{Q}(F^{0})$ is level-wise surjective.
\end{proof}

This is entirely analogous to the tangent complex of the derived flag variety, which is $(\uEnd(V)/\uEnd_{F}(V))[1]$. We also note 
that the zeroth cohomology group consists of orthogonal endomorphisms of $V$ modulo filtered orthogonal endomorphisms of V. 
This is of course the tangent space in the underived case. 
\begin{rk}
Equivalently we can of course write the tangent complex as the shifted cone on the map 
\[
q: \uEnd V/\uEnd_{F}V \to \uHom(\Sym^{2}V,A[2n])/\uHom(\Sym^{2}V/{\cat S_{F}}, A[2n])
\] 
that is induced by $Q$. 
\end{rk}

As the tangent complex and thus the cotangent complex is perfect and as $\pi_{0}\pi^{0}(D)$ is a scheme of finite presentation we deduce:
\begin{cor}\label{cor-lfp}
The derived period stack $D_{n}(V,Q)$ is locally almost of finite presentation.
\end{cor}

\subsection{The derived period domain II: Analytically}\label{sect-anadomain}

As the period map and the second Hodge-Riemann bilinear relation are analytic by nature we need to understand the derived analytic period domain. 

To add the Hodge-Riemann bilinear relations we need some extra structure, so we recall that $V = H(X_{s}, \C)$ is equipped with a real structure and the Lefschetz operator $L$. We define the bilinear forms $Q_{k}$ on $H^{k}V$ by $Q_{k}(\alpha, \beta) = Q(L^{n-k}\alpha,\beta)$. 

\begin{thm}\label{thm-perioddomain}
There is a \emph{derived analytic period domain} $U$, an open analytic substack of $D_{n}(V,Q)^{an}$, which is an enhancement of the period domain. 
\end{thm}
\begin{proof}
We first apply the Lurie-Porta analytification functor reviewed in Section \ref{sect-dan} to the geometric stack $D$. This is possible as $D$ is locally almost of finite presentation by Corollary \ref{cor-lfp}.

We need to check that the underived truncation $\pi_{0}\pi^{0}(D^{an})$ is the closed analytic period domain, but this is immediate from Theorem \ref{thm-periodstack} since
analytification commutes with truncation.

We note that orthogonality with respect to $Q_{k}$ follows from orthogonality with respect to $Q$ as in the definition of $D_{n}(V,Q)$

Now we construct the open substack $U$ of $D^{an}$.
Lemma \ref{lemma-open}
 shows that it suffices to construct an open substack $U'$ on the underlying underived stack. We take $U'$ to be the preimage under the natural map
 $\pi^0(D^{an}) \to \pi_{0}\pi^{0}(D^{an})$ of the classical period domain that is defined inside $\pi_{0}\pi^{0}(D^{an}) = (\pi_{0}\pi^{0}(D))^{an}$ by the second and third Hodge-Riemann bilinear relation. Explicitly, it is the open subspace given by the conditions that $F^{i}H^{k}\oplus \overline {F^{k-i+1}H^{k}} = H^{k}$ and the form given by $(-1)^{k(k-1)/2}i^{p-q}Q_{k}(\cdot, \bar \cdot)$ is positive definite. 
 \end{proof}
\begin{rk}
As $U$ is open in $D^{an}$ it will have the same tangent complex. Moreover, the tangent complex is unchanged by analytification, see the discussion in Section \ref{sect-derivative}.
\end{rk}

Let us now consider how we will construct a map into this derived period domain.

The
period map is naturally a map into a moduli space, however for technical reasons we have now constructed the derived period domain as an analytification of an algebraic moduli stack, rather than as an analytic moduli stack.

Thus we need to consider the following question:
Given the algebraic moduli stack $D$, what can we say about $D^{an}$? It seems too optimistic to expect an analytification of a moduli stack to always represent some corresponding analytic moduli problem. In the setting of this paper we directly construct a map to the analytified derived period domain.

There are several interesting instances where the analytification of a moduli space is the corresponding analytic moduli space. 
For example, in the underived setting one may observe that the analytification of the Grassmannian is an analytic Grassmannian. One way to see this is that there are matching explicit constructions. 

In fact, something very similar holds in our situation and $U$ turns out to be the derived analytic moduli stack of filtrations of $V\otimes \cat O$ by perfect complexes that satisfy the Hodge-Riemann bilinear relations. 

We prove this result in joint work with Mauro Porta as an application of the main theorem of \cite{HolsteinF}, which says
that in many cases the analytification functor commutes with the mapping space functor.

\subsection{Monodromy and quotient}\label{sect-monodromy}
The period map only maps to the quotient of the period domain by monodromy, so we examine monodromy for families of algebraic varieties in the derived case, and then construct the quotient of the derived period domain by a group acting on $V$.

Assume we are given a smooth projective map $f: X \to S$ of derived schemes, i.e.\ a strong map of derived schemes such that $\pi_{0}\pi^{0}(f)$ is smooth and projective. We write $t(S)$ for the underlying topological space of $S^{an}$, i.e. $\pi^{0}S(\C)$ considered in the analytic topology.

The fundamental group of $t(S)$ acts on cohomology of the fibre as each $R^{i}f_{*}\Om_{X/S}$ is a local system on $S$. 
Looking instead at the complex $Rf_{*}\C$ we see that it forms a homotopy locally constant sheaf on $t(S)$. We will see that $Rf_{*}\Om^{an}_{X/S} \simeq Rf_{*} {\set C} \otimes \cat O^{alg}_{S^{an}}$ in Lemma \ref{lemma-trivialize}.

A homotopy locally constant sheaf is equivalently a representation of the simplicial loop group $\Om t(S)$, see \cite{Holstein2}, i.e.\ a priori there is higher monodromy.
But the degeneration of the Leray-Serre spectral sequence for families of projective varieties suggests that $Rf_{*}\C$ is just a direct sum of local systems, i.e.\ there is no higher monodromy.
Indeed, there is the following result due to Deligne: 

\begin{thm}\label{thm-deligne}
Consider a triangulated category $\cat D$ with heart $\cat A$. Let $n \in \set N$, $V \in D(\cat A)$ and assume there is $u: V \to V[2]$ inducing isomorphisms: 
\[u^{i}: H^{n-i}(V) \simeq H^{n+i}(V)\]
for $i \geq 0$. Then there is a canonical map $\phi$ from $HV \coloneqq \sum_{i}H^{i}(V)[-i]$ to $V$ inducing the identity on cohomology. 
\end{thm}
If $\cat D$ is the derived category of $\cat A$ the existence of the isomorphism is Theorem 1.5 in \cite{Deligne68}, but the proof does not depend on this assumption, see also the beginning of \cite{Deligne94}. The canonical map is constructed in \cite{Deligne94}.

Now we let $\cat A$ be the category of sheaves of abelian groups on $t(S)$, $X$ be $Rf_{*}{\set C}$
and $u$ be the map induced by the relative Lefschetz operator. We obtain the following (see \cite{Deligne68} 2.6.3):
 The homotopy locally constant sheaf $Rf_{*}{\set C}$ is quasi-isomorphic to a direct sum of its cohomology sheaves, which are local systems.

We also need to consider the shifted bilinear forms on $Rf_{*}\C$ and its cohomology. We may compute $Rf_{*}(\C)$ in such a way (for example using a \v Cech cover of $t(X)$ over $t(S)$) that it comes with a shifted bilinear form $Q_{R}$ induced by the pairing $\langle \alpha,\beta \rangle = \int \alpha \cup \beta$ 
given by the cup product and the trace map $\int: Rf_{*}\C[-2n] \to R^{2n}f_{*}\C\to \C$ given by Verdier duality. 
This of course induces the pairing $Q$ on cohomology.
$Q_{R}$ also pulls back via $\phi$ to a bilinear form on cohomology $H(Rf_{*}\C)$ and as $\phi$ is the identity on cohomology this is again the usual pairing on cohomology. Thus we have a quasi-isomorphism $\phi: (H(Rf_{*}\C), Q) \to (Rf_{*}\C, Q_{R})$ of complexes of sheaves with shifted bilinear forms.
 
 We deduce:
 
\begin{cor}\label{cor-formality}
The complex $Rf_{*}{\set C}$ with its bilinear form
considered as a $\Om t(S)$-representation is equivalent to a representation of $\pi_{1}(t(S))$. 
\end{cor}

Note that in all this the base 
 is allowed to be singular, the only assumption is that $t(S)$ is locally paracompact.

When constructing the global period map we will see (cf. Lemmas \ref{lemma-trivialize} and \ref{lemma-period})
that this monodromy is the only obstruction to gluing the local derived period maps, i.e.\ the derived structure of $S$ (which is infinitesimal) does not
interfere.
\begin{rk}
On the other hand Voisin shows in \cite{Voisin11}
that Deligne's decomposition does not work on the level of algebras. This suggests that extending the period map to keep track of the algebra structure on $Rf_{*}\Om_{X/S}$ poses a significantly harder problem. The bilinear form is defined using the algebra structure, but we have seen that preserving the bilinear form is a much weaker condition than preserving the algebra structure.
\end{rk}

\begin{rk}
One might wonder what happens if one considers the complex $Rf_{*}\Om_{X/S}$ in the case that $S$ is underived, i.e.\ if one considers a classical period map taking values in complexes.

Consider the Hodge filtration $F^{i}$ on $Rf_{*}\Om_{X/S}$. Then we note that now all the $F^{i}$ are formal, i.e.\ quasi-isomorphic to $HF^{i}$. This follows since they have locally free cohomology. (In the derived setting this argument would not apply.) 
So the algebraic geometry of $S$ is not reflected in studying Hodge structures as a complex (as opposed to a collection of cohomology groups), and the topology is not reflected beyond the fundamental group.

This means for example that failures of the Torelli theorem are not resolved by considering a period map taking values in complexes.
\end{rk}

We will now connect the monodromy considerations to our construction of $U$. We think of $V = V_{\Z}\otimes \C$ as the cohomology complex of a fibre $X_{s}$ of our map $X \to S$ of derived schemes. (Note $X_{s}$ is a classical smooth projective scheme.)
Then we have a strict action of the fundamental group $\pi_{1}(S)$ on $V$ compatible with the form $Q$.
In fact this action factors through the action of the universal arithmetic group $\Gamma = \operatorname{Aut}(V_{\Z},Q)$ on $V$ with the bilinear form $Q$.

We will now show that the action of $\Gamma$ on $V$ induces an action on the derived period domain.

\begin{propn}\label{propn-quotientspace}
 The action of $\Gamma$ on $V$ induces an action on $D_{n}(V,Q)$ which restricts to an action on the derived period domain $U$ and the homotopy quotient $U/\Gamma$ exists as a geometric derived analytic stack. 
\end{propn}
\begin{proof}
As $V$ comes with a $\Ga$-action there is a map $B\Ga \to Perf$ classifying this representation and $V: * \to Perf$ factors through it. 
Note that the simplicial set $B\Ga$ can be considered as a constant stack (i.e. we stackify the constant functor $B\Ga$ on derived schemes). We do not expect this to be geometric or locally almost of finite presentation as $\Ga$ is an infinite discrete group.
As the action is compatible with $Q$ we can consider the map $\gamma: B\Ga \to QPerf$ classifying the pair $(V,Q)$.

Then the pullback of $QFilt \to QPerf$ along $\gamma$ is a derived stack $D_{\Gamma}$. As the homotopy pullback of $D_{\Gamma}\to B\Gamma$ along the universal cover $* \to B\Ga$ is equivalent to $D$ (see Remark \ref{rk-firstdescription}) we see that $D_{\Gamma}$ is the desired quotient. 
Explicitly, $\Ga$ acts on $D_{h} \coloneqq D_{\Ga}\times_{B\Ga}E\Ga $ via its action on $E\Ga$, and of course $D_{h} \simeq D$. Analytification preserves effective epimorphisms 
as it is defined as a left Kan extension along a continuous morphisms of sites. It also commutes with homotopy limits and thus we can consider $D_{\Ga}^{an}$ as a quotient of $D^{an}$. 

The underlying underived stack is not necessarily geometric, but we will now construct a geometric homotopy quotient  of $U$ by $\Gamma$. 

As the non-degeneracy condition defining $U$ (see Theorem \ref{thm-perioddomain}) is unaffected by $\Gamma$ it defines a substack $U/\Gamma$ of $D^{an}_{\Gamma} \simeq D^{an}_{h}/\Gamma$. We also note that the open substack $U\subset D^{an}$ corresponds to an open substack $U_{h}$ of $D_{h}^{an}$, which is preserved by the action of $\Gamma$, as this is true for the underlying underived spaces where $\Gamma$ is the usual monodromy action.
$U/\Gamma$ is equivalent to the homotopy quotient $U_{h}/\Gamma$, justifying the name.
Now we use the fact that the action of $\Gamma$ on $\pi_{0}\pi^{0}(U_{h})$ (which is just the usual period domain) is properly discontinuous to show that $U/\Gamma$ is geometric.  
 
There is an open cover $\mathfrak U$ of the period domain $\pi_{0}\pi^{0}(U^{h})$ such that the images of any two open sets under the group action intersect only finitely many times. This gives a cover of the stack $\pi^{0}(U^{h})$ and thus of the derived stack $U^{h}$ with the same property. 
Thus $U_{h}$ provides a cover of $U_{h}/\Gamma$ that is locally a quotient by a finite group, which is enough to show $U/\Ga$ is geometric. 
\end{proof}

\section{The period map}\label{sect-map}

\subsection{The classical period map} 
In this section, which is only used as background, we recapitulate Griffiths' period map \cite{Griffiths70}.

For a polarized family of smooth projective varieties $f: X \to S$, where $S$ of finite type, there are holomorphic maps $P_{k}: S^{an} \to U/\Ga$, where $U \subset Flag(F^{*},H^k)$ is the polarized period domain and $\Ga$ is the  monodromy group. Here $P_{k}(s)$ is $F^{n}H^{k} \subset \dots \subset F^{1}H^{k}(X_{s}) \subset H^{k}(X_{s}) \cong H^{k}(X_{0})$.

We define $P(s)$ as the product of the Hodge filtrations on $H^{k}(X_{s})$ for all $k$.
 
We give some details: $\Om^{1}_{X/S}$ is a vector bundle as $f$ is assumed smooth. Then the relative de Rham complex is given by exterior powers, with a differential which is not $\cat O_{X}$-linear, 
hence this is not a complex of coherent sheaves. However, $d$ is $f^{-1}\cat O_{S}$-linear and thus we have a complex of $f^{-1}\cat O_{S}$-modules, which pushes forward along $f$ to a complex of $\cat O_{S}$-modules.

Note that Deligne shows that the relative de Rham complex $\Om^{an}_{X/S}$ is quasi-isomorphic to $f^{-1}\cat O_{S^{an}}\otimes \C$ and uses this to show that $Rf_{*}\Om^{an}_{X/S} \simeq Rf_{*}(\C) \otimes \cat O_{S^{an}}$, see Proposition 2.28 of \cite{Deligne70}. This works if the base is any analytic space, and if $f$ is just assumed smooth and separated.

By Ehresmann's theorem we have diffeomorphisms $X_{s} \to X_{0}$ for every path from $0$ to $s$, showing the cohomology sheaves are locally constant and if $S^{an}$ is simply connected $R^{i}f_{*}\Om^{an}_{X/S}$ becomes $H^{i}(X_{s},\C)\otimes \cat O_{S^{an}}$.

The stupid truncation $\cat F^{p} = \Om_{X/S}^{\geq p}$ is a subcomplex of sheaves of $f^{-1}\cat O_{S^{an}}$-modules. 
One then shows that the $R^{i}f_{*}\Om^{an}_{X/S}$ are vector bundles (and in particular coherent sheaves). This implies that the $R^{i}f_{*}\cat F^{p} = R^{i}f_{*}((\Om^{an}_{X/S})^{\geq p})$ are complexes of coherent sheaves. This can be done by trivialising the fibration locally and using Grauert's base change theorem, see \cite{Peters08}.

The $R^{i}f_{*}\cat F^{p}$ are moreover subsheaves of the $R^{i}f_{*}\Om^{an}_{X/S}$.
This follows from degeneration of the Fr\"ohlicher spectral sequence as the differentials in the $R^{i}f_{*}$ long exact sequence coming from $\Om^{\geq p}\to \Om^{\geq p-1} \to \Om^{p-1}$ vanish. Degeneration of the spectral sequence follows from a dimension count.

Thus the map $P_{k}$ sending $S^{an}$ to the filtration $\{R^{k}f_{*}\cat F^{i}\}$ of $H^{k}(\Om^{an}_{X_{0}})$ is the desired map to the flag variety and it factors through the period domain. 

Finally, we can globalize the construction by dividing out by the action of the fundamental group of $S^{an}$. We use Ehresmann's theorem again to pull all our data along any path in the base in a homotopy invariant way, giving an action of the fundamental group of $S^{an}$.

\begin{rk}
These results are often stated for $S$ smooth, but we need not assume this. Deligne's $\Om^{an}_{X/S} \simeq f^{-1} \cat O_{S^{an}}\otimes \C$ holds over an analytic space, as do Grauert's theorem (see \cite{Grauert60}),
and Ehresmann's theorem (see Demailly's chapter in \cite{Bertin02}). 

Grauert's theorem assumes that $S$ is reduced (there is a version for a non-reduced base). But we may deal with non-reduced bases as part of our proof for the derived Poincar\'e Lemma, see Lemma \ref{lemma-poincare2} below.
\end{rk}

The map $P$ is a priori not algebraic. In fact, to construct the map we needed to locally trivialize $Rf_{*}\set C \otimes \cat O_{S}$.
In the analytic setting we can do this over any contractible set, so we just need $S^{an}$ to be locally contractible. In the algebraic setting this is typically impossible unless $S$ is itself simply connected.

Finally, let us have a look at the differential. Assume for simplicity that $S$ is smooth.

The differential of the period map $dP_{p,k}$ factors through the Kodaira-Spencer map $B \to H^{1}(X, \cat T_{X})$.
In fact $dP$ is the composition of Kodaira-Spencer with the map
\[H^{1}(X, \cat T_{X}) \to \Hom(F^{p}H^{k}(X), H^{k}(X)/F^{p}H^{k})\]
which is given by composing the natural cup product and contraction map with the natural quotient and 
inclusion maps: $F^{p}H^{k} \to H^{p,k-p}(X) \to H^{p+1,k-p-1}(X) \to H^{k}/F^{p}H^{k}(X)$.

The period map satisfies Griffiths transversality, i.e.\ the image of the differential is contained in the subspace $F^{p-1}H^{k}(X)/F^{p}H^{k}(X)$. 
This can also be expressed as saying the connection on $Rf_{*}\Om^{*}_{X/S}$ maps $F^{i}$ to $F^{i-1}\otimes_{\cat O_{S}} \Om^{1}_{S}$.

\subsection{The local derived period map I: The de Rham complex}\label{sect-localperiod-1}
We now define the derived version of the period map. In this section we work with derived schemes modelled on simplicial commutative algebras. 

We will consider a polarized smooth projective map $f: X \to S$ between geometric derived stacks. We fix the fibre $X_{s}$ over some distinguished point $s$. This is a smooth projective variety. 
We assume that $S$ is of finite presentation.

\begin{rk}
We need our base stack to be locally almost of finite presentation to use analytification, and we need that it is quasi-compact to globalise our construction in Theorem \ref{mainA}.
\end{rk}

By definition the map $f$ is smooth and projective if it representable, 
$\pi^{0}(f)$ is a smooth and projective map of algebraic stacks and $f$ is strong, i.e. $\pi_{i}(\cat O_{X}) \cong \pi_{i}(\cat O_{S}) \otimes_{\pi_{0}(\cat O_{S})} \pi_{0}(\cat O_{X})$. Recall that smoothness of algebraic stacks can be defined locally as in Definition 1.3.6.2 of \cite{Toen05a}.
By polarized we just mean that $\pi^{0}(f)$ is polarized.

We will first consider the case that $S$ is a quasi-separated derived scheme. We note that the definition of a derived scheme $X$ includes that it has an underlying topological space, the same is true for its analytification. In both cases we denote the underlying topological space with the analytic topology by $t(-)$.

We will construct a map from open pieces of the analytification $S^{an}$ of $S$ into the derived period domain $U \subset D_{n}(V,Q)^{an}$, constructed in Theorem \ref{thm-perioddomain}.
Explicitly, we take $V$ to be the complex $H^{*}(X_{s}, \Om_{X_{s}})$ with zero differentials. 
For $Q$ we take the usual shifted bilinear form defined using product and trace.

As we can only hope to construct the map locally we will now consider smooth maps $i_{T}: T \to S^{an}$ such that $\pi^{0}(T)$ is a contractible Stein space. 
Locally this is always possible as every point in a complex space has a contractible Stein neighbourhood, see e.g.\ Theorem 3.1 in \cite{Mihalache96}.
We may also pick these neighbourhoods to be relatively compact in $S^{an}$ (i.e.\ with compact closure). 
We write $j_{T}:X_{T} \to X^{an}$ for the pullback of $X^{an}$ to $T$. This is in particular a smooth map. 

More generally, if $S$ is a stack we pick a good hypercover $T_{\bullet}$ of $S^{an}$, 
i.e.\ an effective epimorphism $T_{\bullet} \to S^{an}$ such such that all $t(T_{n})$ are disjoint unions of contractible spaces. This is possible by picking an affine hypercover first, analytifying and then using again that any Stein space is covered by contractible Stein spaces. Using Proposition 7.2.1.14 in \cite{Lurie11e} we may check that the map is an effective epimorphism by checking the truncated map is an epimorphism.
Then we have $S^{an} \simeq \hocolim_{I} T_{i}$ by Lemma 3.6 in \cite{Porta15}. 

We will take care to perform our constructions compatibly with base change so that we can glue together the derived period map in Section \ref{sect-globalperiod}.

So let us consider the local situation. We will often consider the affinization of $T$, i.e.\ the derived affine scheme $R\Spec(\cat O^{alg}(T))$ and denote it by $uT$.

Abusing notation we will use $f$ also for $f^{an}$ and $f^{an}|_{X_{T}}$ when there is no ambiguity.

Now by Theorem \ref{thm-periodstack} 
one way to define a map $T \to U$ would be to have a map $S \to D_{n}(V,Q)$ such that its analytification satisfies the positivity conditions. Of course as long as $S$ is algebraic this is not realistic, but we will construct a map $S \to QFilt$ as one ingredient in our construction.

In what follows we need to explicitly consider coherent sheaves on derived analytic spaces like $T$. Recall from Section \ref{sect-dan} that these are just $\cat O_{T}^{alg}$-modules such that the cohomology sheaves are locally coherent over $\pi_{0}(\cat O_{T}^{alg})$.
Moreover the underlying $\oo$-topos is the $\oo$-topos of $t(T)$. To compute derived functors we use the local model structure on presheaves. On stacks we use homotopy cartesian sections for a cover as described in Section \ref{sect-alggrass}.

In particular we can apply the usual direct and inverse image functors between sheaves of complexes, on $t(T)$. There is extra structure on the space of functions on a derived analytic stack, but it will not affect the constructions we need.

The crucial ingredient for Hodge theory is the relative de Rham complex.
This is defined as the exterior algebra of the relative cotangent complex (which is a simplicial sheaf). 

The relative cotangent complex should be available in every good theory of derived analytic spaces and treatment of the analytic cotangent complex can now be found in \cite{Porta16}. For our purposes we will avoid the analytic theory and just use the analytification of the algebraic cotangent complex. (At least for Deligne-Mumford stacks this agrees with the analytic cotangent complex, see Theorem 5.20 of \cite{Porta16}.)

Thus in our set-up we may use the analytification of the relative algebraic cotangent complex. As $f$ is smooth we think of this as the complex of K\"ahler differentials. So we are considering $\Om^{1}_{X/S} = \Om^{1}_{X}/f^{*}\Om^{1}_{S}$ and restrict to $X^{an}_{T}$; we define $\Om^{1}_{X/T} \coloneqq j_{T}^{*}((\Om_{X/S}^{1})^{an}) = j_{T}^{*}h_{X}^{*}\Om_{X/S}^{1}$ for any inclusion $j_{T}: X_{T}^{an} \to X^{an}$ where $h_{X}$ is the natural map $X^{an}\to X$.

For definitions of the algebraic cotangent complex $
\set L$, see \cite{Toen05a}. We recall some basic local facts:
Let $f^{*}: A \to B$ be a map of simplicial algebras. The relative cotangent complex $\set L_{B/A}$ fits into an exact sequence  $B \otimes^L_{A} \set L_{A} \to \set L_{B} \to \set L_{B/A}$.  It satisfies base change, see Proposition 1.2.1.6 in \cite{Toen05a}.
For an explicit construction replace $B$ by an algebra that is free over $A$, see Section 4 of \cite{Goerss06}.
 In particular then each $B_{n}$ is a free algebra over $A_{n}$.
 We know that $\set L_{A}$ is the simplicial $A$-module that is given by the K\"ahler differentials $\Om^{1}_{A_{n}}$ in degree $n$. As $B$ is free over $A$ we have cofibrations of $B_{n}$-modules $B_{n} \otimes_{A_{n}} \Om^{1}_{A_{n}} \to \Om^{1}_{B_{n}}$ and as the construction of K\"ahler differentials is functorial in pairs of algebras we find that $\set L_{B/A}$ is the simplicial $B$-module that is $\Om^{1}_{B_{n}}/(B_{n}\otimes_{A_{n}} \Om^{1}_{A_n}) = \Om^{1}_{B_{n}/A_{n}}$ in degree $n$. 
 We write $\Om_{B_{n}/A_{n}}^{i}$ for $\wedge^{i}\Om^{1} _{B_{n}/A_{n}}$. 
If $\Spec(B)$ over $\Spec A$ is smooth of dimension $n$ then $\Om^{n+1} \simeq 0$ canonically and we may modify $\Om^{*}$ to assume this is an equality. We shall do so from now on. 
For a smooth morphism $\set L_{B/A}$ is perfect, see Lemma 2.2.2.5 in \cite{Toen05a}.

We can thus define a chain complex of simplicial modules which is $\Om^{i}_{B_{n}/A_{n}}$ in bidegree $(i,n)$.
Explicitly we have the following. Here each $d$ is $A_{i}$-linear.
\[
\begin{matrix}
A_{0} &\stackrel {f^{*}_{0}} \to& B_{0} & \stackrel d \to & \Om^{1}_{B_{0}/A_{0}} &\stackrel d \to & \Om^{2}_{B_{0}/A_{0}} &\to \dots \\
\uparrow \uparrow & & \uparrow \uparrow & & \uparrow \uparrow & & \uparrow \uparrow \\
A_{1} &\to& B_{1} & \to & \Om^{1}_{B_{1}/A_{1}} &\to & \Om^{2}_{B_{1}/A_{1}} &\to \dots \\
\uparrow \uparrow \uparrow & & \uparrow\uparrow \uparrow & & \uparrow \uparrow\uparrow & & \uparrow \uparrow\uparrow \\
A_{2} &\to& B_{2} & \to & \Om^{1}_{B_{2}/A_{2}} &\to & \Om^{2}_{B_{2}/A_{2}} &\to \dots \\
\vdots & & \vdots & & \vdots & & \vdots \\
\end{matrix}
\]
We denote the total complex of the associated double complex of $N(A)$-modules by  $\Om_{B/A}$. 
\begin{rk}
Note that we are indeed taking the exterior algebra in every degree, rather than taking the symmetric algebra in odd degrees, cf. \cite{Illusie06}. 
\end{rk}
\begin{rk}
A deeper study on differential forms in derived algebraic geometry is done in \cite{Pantev11}.
However, for a smooth morphism we do not have to worry about the subtleties dealt with in that work, cf.\ Section 1.2 of \cite{Pantev11}.
\end{rk}

\begin{lemma}\label{lemma-derham}
The cotangent complex $\Om^{1}_{X/T}$ defined above is a coherent sheaf on $X^{an}$.  Restricting to the underived setting we recover the usual cotangent complex for a smooth map, i.e.\ the sheaf of relative holomorphic differentials.
\end{lemma}
\begin{proof}
The first claim is true as analytification preserves coherence, see \cite{Porta15}.

Next note that holomorphic differentials are the analytification of the sheaf of K\"ahler differentials, and derived analytification is compatible with truncation, see Section 7 of \cite{Porta15}.
\end{proof}

For the next lemma we will consider the double chain complex of sheaves associated to the simplicial chain complex $(\Om_{X/T}^{*})_{\bullet}$. We will denote its total complex by $\Om_{X/T}$. 
This is a complex of $N(f^{-1}\cat O_{S})$-modules.

When $f$ is smooth we have the following derived Poincar\'e lemma:

\begin{lemma}\label{lemma-poincare2}
Let $T$ be  
an open subset of the analytification of a derived scheme. Then $\Om_{X/T}$ as constructed above is a resolution of $N(f^{-1}\cat O_{T}^{alg})$.
\end{lemma}
\begin{proof}
The question is local, so let us first assume $X \to S$ is a smooth map of derived affine schemes, and let $i_{T}: T \to S^{an}$ be an open subspace such that $\pi^{0}(T)$ is contractible Stein and relatively compact. 
We consider the map $A = \cat O(S) \to B=\cat O(X)$ of simplicial algebras.

There clearly is a map $e: N(A) \to \Om^{*}_{B/A}$ of $N(A)$-modules. We claim it is a quasi-isomorphism after base-change to $\cat O(T)$, i.e.\ after analytification and restriction to $T$. We note that $i_{T}: T \to S^{an}$ is flat as $T$ is relatively compact in $S^{an}$, see Lemma 8.13 of \cite{Porta14}.
Then $\iota = h_{X}\circ i_{T}$ is flat since $h_{X}$ is flat, see Corollary 5.15 of \cite{Porta15}. So we aim to show that $\iota^{*}(e)$ is a quasi-isomorphism.

We  consider $\mathfrak n$, the ideal of nilpotents and elements of positive degree in $N(A)$. By abuse of notation we will use the same name for the corresponding ideal in $A$.

This is a simplicial ideal and $A/\mathfrak n$ is reduced and underived. 
We also consider the ideal $\mathfrak n_{B}=\mathfrak n \otimes_{A}B$ in $B$. 
Note that $B \otimes_{A} A/\mathfrak n$ is also reduced and constant (as a simplicial algebra) as we assume $f$ smooth. 

The $\mathfrak n^{i}$ form an exhaustive filtration of $N(A)$. 
As $\Om^{*}_{B/A}$ is an $N(A)$-complex we can filter it by the $\mathfrak n^{i}$, too. (As $\Om^{*}_{B/A}\otimes_{N(A)}\mathfrak n^{i}$ are subcomplexes.) We aim to show that the associated graded map of $\iota^{*}(e)$ induces a quasi-isomorphism, which shows that $\iota^{*}(e)$ itself is a quasi-isomorphism. As $\iota$ is flat we can equivalently show that the associated graded pieces of $e$ become quasi-isomorphisms after applying $\iota^{*}$.

Write $A'$ for $A/\mathfrak n$ and $B'$ for $B/(\mathfrak n \otimes_{A} B)$.
Now we claim the following: 
\begin{itemize}
\item $N(A') = N(A)/\mathfrak n$ and $\Om^{*}_{B'/A'}$ become quasi-isomorphic when applying $\iota^{*}$,
\item $Gr_{\mathfrak n}^{i}N(A)$ and $ \Om^{*}_{B'/A'} \otimes_{N(A')} Gr_{\mathfrak n}^{i}N(A)$ become quasi-isomorphic when applying $\iota^{*}$,
\item $Gr^{i}_{\mathfrak n}(\Om^{*}_{B/A}) \simeq \Om^{*}_{B'/A'} \otimes_{N(A')} Gr_{\mathfrak n}^{i}N(A)$.
\end{itemize}

The first claim follows from the usual underived Poincar\'e lemma, see for example Deligne \cite{Deligne70}, observing that $\Spec(B')^{an}\to \Spec(A')^{an}$ is a smooth map of analytic spaces. Then we base change to $T_{red}$ and apply the exact global sections functor. 

The second claim follows from the first since $\Om^{*}_{B'/A'}$ is flat over $N(A')$ and pullback is compatible with tensor products.

So let us turn to the third claim. Here we will have to work simplicially. We begin with $i=0$ and show that $\Om^{m}_{B/A}/\mathfrak n \cong \Om^{m}_{B'/A'}$.
Recall that $\Om^{m}_{B/A}$ is $\Om^{m}_{B}/\Om^{m-1}_{B} \otimes_{A} \Om^{1}_{A}$. So we can consider the natural map 
\[
\Om^{m}_{B} \stackrel g \to \Om^{m}_{B'} \to \Om^{m}_{B'}/(B' \otimes_{A'}\Om^{1}_{A'})\]
The map is clearly surjective and we claim the kernel is given by $\mathfrak n \otimes_{A}\Om^{m}_{B} \oplus \Om^{m-1}_{B} \otimes_{A} \Om^{1}_{A}$.

It follows from the short exact sequence $\mathfrak n/\mathfrak n^{2}\to \Om^{1}_{B} \to \Om^{1}_{B'}$ that the kernel is given by $d\mathfrak n_{B} + g^{-1}(\Om^{i-1}_{B'} \otimes_{A'}\Om^{1}_{A'}) = \mathfrak n \otimes \Om^{m}_{B} + \Om^{1}_{A}\otimes \Om^{m-1}_{B}$ as $d\mathfrak n \subset \Om^{1}_{A}$.

So we have shown $\Om^{m}_{B/A}/\mathfrak n \cong \Om^{m}_{B'/A'}$ as simplicial $A'$-modules. The differential is compatible thus we get a quasi-isomorphism with the associated complexes of $N(A')$-modules, $\Om^{*}_{B/A}/\mathfrak n \simeq \Om^{*}_{B'/A'}$.

To consider the other associated graded pieces we observe that for any $A$-modules $M$ we have $(I^{n}A\otimes_{A} M)/(I^{n+1}A \otimes_{A}M) \cong (I^{n}A/I^{n+1}A)\otimes_{A/I} M/I$. This is valid in the simplicial setting as it holds degree by degree. Again we deduce the quasi-isomorphism of complexes.

Together these claims give the desired quasi-isomorphism on the associated graded modules and we have locally shown the quasi-isomorphism $N(f^{-1}\cat O_{T}^{alg})\simeq \Om^{*}_{X/T}$ of sheaves. 
\end{proof}
We could now consider the derived push-forward $Rf_{*}\Om_{X/T}$ and it is not hard to deduce that $Rf_{*}\C \otimes \cat O_{T}^{alg} \simeq Rf_{*}\Om_{X/T}^{an}$. However, to define the bilinear form it is more convenient for us to forget the differentials and work with $\Om^{\#}_{X/T} \coloneqq (\Om^{*}_{X/T}, 0)$.
Thus, instead of using Lemma \ref{lemma-poincare2} we deduce the main consequence of the Poincar\'e Lemma together with Hodge-to-de Rham degeneration from the classical case as follows:
\begin{lemma}\label{lemma-hodge2derham}
Let $\Om^{\#}_{X/T} = \oplus \Om^{i}_{X/T}[-i]$ be the de Rham algebra without its differential. Then $Rf_{*}f^{-1}\cat O^{alg}_{T} \simeq Rf_{*}\Om^{\#}_{X/T}$. 
\end{lemma}
\begin{proof}
There is a natural map induced by the inclusion $f^{-1} \cat O^{alg}_{T}\to \Om^{\#}_{X/T}$ and we need to check it is a quasi-isomorphism.
As $f$ is smooth, $\Om^{i}_{X/S}$ is a projective $\cat O_{X}$-module \cite[Definition 1.2.7.1]{Toen05a} 
and in particular strong \cite[Lemma 2.2.2.2]{Toen05a}, 
thus, after analytifying and restricting, $H^{k}(\Om^{i}_{X/T}) \cong H^{k}(\cat O_{X}) \otimes_{H^{0}(\cat O_{X})} H^{0}(\Om^{i}_{X/T})$. 
Moreover, $\cat O_{X}$ is a strong $f^{-1}\cat O_{S}$-module and we have $H^{k}(\Om^{i}_{X/T}) \cong H^{k}(f^{-1}\cat O_{S}) \otimes_{H^{0}(f^{-1}\cat O_{S})} H^{0}(\Om^{i}_{X/T})$. We are in the analytic setting but do not need to worry about completing the tensor product as we work with finitely presented modules.

The $E_{2}$-term of the natural Grothendieck spectral sequence computing $R^{k}f_{*}\Om^{\#}_{X/T}$ is thus $R^{p}f_{*}(H^{0}(\Om^{i}_{X/T}) \otimes H^{q}(f^{-1}\cat O^{alg}_{T}))$.
As $f$ is proper  we apply the projection formula to the complexes $\Om^{i}_{X/T}$ and $f^{-1}\cat O^{alg}_{T}$ to find $E_{2}^{pq} \cong R^{p}f_{*}(H^{0}(\Om^{i}_{X/T})) \otimes H^{q}(\cat O^{alg}_{T})$.  For the projection formula one usually assumes bounded below complexes of sheaves (with the cohomological grading) and the structure sheaf is not bounded below in general, but because the fibres of our map of topological spaces are locally finite-dimensional the theorem remains true for unbounded complexes, see \S 6 of \cite{Spaltenstein88}.

Similarly on the left hand side we have a spectral sequence with $E_{2}^{pq} = R^{p}f_{*}(H^{0}(f^{-1}\cat O^{alg}_{T})) \otimes H^{q}(\cat O^{alg}_{T})$.
The natural map induces an isomorphism since $H^{0}(f^{-1}\cat O^{alg}_{T}) \simeq H^{0}(\Om_{X/T})$ and $Rf_{*}H^{0}(\Om_{X/T})\simeq Rf_{*}H^{0}(\Om^{\#}_{X/T})$ is just classical Hodge theory.
\end{proof}

Using the product on $\Om^{\#}_{X/T}$ and suitable resolutions we may assume this is a sheaf of algebras.
\begin{lemma}\label{lemma-trivialize}
Let $f: X \to S$ be a smooth projective morphism of derived schemes and consider an open derived Stein $T \subset S^{an}$ 
and $\Om^{\#}_{X/T}$ as above. Assume moreover that $t(T)$ is simply connected. 

Then there is a quasi-isomorphism of $\cat O_{T}^{alg}$-algebras $\bar v: Rf_{*}\C\otimes \cat O_{T}^{alg} \to Rf_{*}(\Om^{\#}_{X/T})$ and a quasi-isomorphism of $\cat O_{T}^{alg}$-modules
 $\varphi: V \otimes \cat O_{T}^{alg} \to Rf_{*}\C \otimes \cat O_{T}^{alg}$. 
We write $\bar w$ for $\bar v \circ \varphi$. These maps are compatible with smooth base change.
\end{lemma}
\begin{proof}
By Lemma \ref{lemma-hodge2derham} we have a weak equivalence $Rf_{*}f^{-1}\cat O_{T}^{alg}\to Rf_{*}\Om^{\#}_{X/T}$.
As $f$ is proper with locally finite dimensional fibres we may use the projection formula to find $Rf_{*}\C \otimes \cat O^{alg}_{T} \simeq Rf_{*}f^{-1}\cat O_{T}^{alg}$. We call the composition $\bar v$. The tensor product here is over the constant sheaf and can be considered underived. 

We may compute the natural map in the category of sheaves of algebras and thus arrange that $\bar v$ is a homomorphism. 
 
Now $Rf_{*}\C$ only depends on the underlying topological space, hence it is a homotopy locally constant sheaf on $t(T)$ by the classical result.
By Corollary \ref{cor-formality} this is actually formal, i.e.\ quasi-isomorphic to the graded local system of its cohomology sheaves. But as $t(T)$ is simply connected this must be a constant sheaf, which is of course quasi-isomorphic to $V$. 
Thus we let $\varphi$ be Deligne's map from \ref{thm-deligne}, tensored with $\cat O_{T}^{alg}$.

We need to check $\bar w$ is compatible with base change, i.e.\ given a smooth
$r: T' \to T$ the quasi-isomorphism $\bar w'$ constructed from $T' \to S^{an}$ is equivalent to $r^{*}\bar w$ under the natural transformations $r^{*}(V \otimes \cat O^{alg}_{T}) \to V \otimes \cat O^{alg}_{T'}$ and 
$r^{*}(Rf_{*}\Om^{\#}_{X/T}) \to Rf'_{*}(\Om^{\#}_{X'/T'})$. These are weak equivalences by base change for the cotangent complex and base change for derived schemes (Proposition 1.4 of \cite{Toen12}). 
The equivalence of $r^{*}\varphi$ and $\varphi'$ is clear as everything is compatible with the Lefschetz operator. Compatibility of $\bar v$ follows from nonabelian base change for locally compact Hausdorff spaces, see Corollary 7.3.1.18 in \cite{Lurie11a}. 
\end{proof}

Our next ingredient is the bilinear form which we construct algebraically. It will follow directly from the existence of the trace map on $f_{*}\Om^{\#}_{X/S}$.
\begin{lemma}\label{lemma-trace}
Given $f: X \to S$ smooth and proper and such that $S$ is derived scheme there is a natural trace map $Rf_{*}(\Om^{\#}_{X/S}) \to \cat O_{S}$. In particular it is compatible with base change.
\end{lemma}
\begin{proof}
This is essentially Grothendieck-Verdier duality.
First we observe that by the adjoint functor theorem under mild conditions there is an exceptional inverse image functor, i.e.\ a right adjoint to $Rf_{*}$, denoted $f^{!}$. See Section 6.4 in \cite{Lurie18}. Secondly, we need a comparison map between
differential forms and $f^{!}\cat O_{S}$.  We construct a map $\Om^{\#}_{X/S} \to f^{!}\cat O_{S}$ following the argument in the underived case given by Neeman in Construction 3.1.5 of \cite{Neeman18}. (Note that for our purposes we do not need to check that this map is a weak equivalence.) 

The construction uses base change for the exceptional inverse image to formally construct a map $L\Delta^{*}R\Delta_{*}\cat O_{S} \to f^{!}\cat O_{S}$ (where $\Delta: S \to S \times S$ is the diagonal). The derived version of the necessary base change theorem is proved in Proposition 6.4.2.1 of \cite{Lurie18}. As $f$ has finite-dimensional fibres it is locally of finite $\Tor$-amplitude.  

To complete the construction one uses the Hochschild-Kostant-Rosenberg theorem identifying $\Om^{\#}_{X/S}$ with $L\Delta^{*}R\Delta_{*}\cat O_{S}$. 
The derived HKR theorem over an affine base can be found in Corollaire 4.2 of \cite{Toen09}. The statement is over an underived base, but the proof works the same way over a derived ring $A$. To globalise from a derived affine scheme $S = R\Spec(A)$ to the case where $S$ is a scheme, and in fact an Artin stack, we may use Proposition 1.14 of \cite{Pantev11} which shows that global sections of the de Rham algebra may be computed as a homotopy limit of the de Rham algebras of affine pieces. It is easier to see that the construction of Hochschild homology globalises similarly: We can check directly that $L\Delta^{*}R\Delta_{*}\cat O_{S}$ is compatible with base change.

Compatibility of the trace map with base change follows as in Reduction 3.2.5 of \cite{Neeman18}, using again the derived base change from \cite{Lurie18}.
\end{proof}

\begin{lemma}\label{lemma-constructq}
Given $f: X \to S$ smooth proper and polarized and such that $S$ is a derived scheme there is a natural non-degenerate shifted symmetric bilinear form $Q_{F}$ on the perfect complex $F^{0}= Rf_{*}(\Om^{\#}_{X/S})$. It reduces to the classical pairing induced by the trace in the non-derived truncation. $Q_{F}$ is compatible with base change.
\end{lemma}
\begin{proof}
We compose the trace map from Lemma \ref{lemma-trace} with the natural product on $F^{0}$ to define $Q_{F}$. By construction this enhances the classical bilinear form.

To check non-degeneracy of $Q_{F}$ it is sufficient to check that the analytification is non-degenerate locally.
We will compare $Q_{F}^{an}$ on $T \subset S^{an}$ to $Q\otimes \cat O^{alg}_{T}$ which is clearly non-degenerate.
In other words we show that $\bar w$ from Lemma \ref{lemma-trivialize} induces a weak equivalence $(V\otimes \cat O^{alg}_{T}, Q \otimes \cat O^{alg}_{T}) \simeq (Rf_{*}\Om^{\#}_{X/T}, Q_{F}^{an})$. The two bilinear forms agree in the reduced case by classical results, see \cite{Sastry03} for a comparison of trace pairing and Hodge pairing. It follows that $\bar w$ preserves the bilinear form on the subcomplex $V$. Since $\bar w$, $Q_{F}^{an}$ and $Q\otimes \cat O^{alg}_{T}$ are all compatible with the $\cat O_{S}$-structure this gives the desired result.

Finally, for base change we need to pick $g: S' \to S$ and compare $g^{*}Q_{F}$ with $Q_{F'}$ where $F'$ is the pushforward of the de Rham algebra on $X \times_{S}S'$. This follows from naturality of the trace map. 
\end{proof}

\subsection{The local derived period map II: The Hodge filtration}\label{sect-localperiod-2}
To  construct  the Hodge filtration we look at the inclusion of sheaves given by the stupid truncation of the relative de Rham complex, $\cat F^{i} \subset \cat F^{i-1} \subset \Om_{X/T}$.

We first consider the Hodge filtration without trivialising the de Rham complex. As our constructions are compatible with base change we may work directly over a stack.
\begin{lemma}\label{lemma-filtmoduli} 
Given $f: X \to S$ smooth proper and polarized and such that $S$ is derived Artin stack. Then the filtered de Rham complex induces a map $\tilde \phi: S \to \cat QFilt$.
\end{lemma}
\begin{proof}
We sketch how to define the pushforward map on stacks. (See the proof of Proposition 2.1.1 in \cite{Gaitsgory11} for the $\oo$-categorical analogue.)

We choose an affine hypercover $\{S_{i}\}_{i\in I}$ of $S$ which induces a cover $\{X_{i} = X\times_{S^{an}}S^{an}_{i}\}$ of $X^{an}$. All $X_{i}$ are derived schemes as $f$ is smooth. Now the pushforward maps on the $X_{i}$ give a pushforward map of presections.
The pushforward is a homotopy cartesian section by base change for quasi-compact and quasi-separated maps of derived schemes, see Proposition 1.4 of \cite{Toen12}. 

Base change also implies that the pushforward construction is well-defined. Given two different affine covers $S_{\bullet}$ and $S'_{\bullet}$ of $S^{an}$ we may find a common refinement by taking a hypercover of the fibre product. 
Thus it is enough to check that the construction is invariant under refinements of the cover $u: S_{\bullet}' \to S_{\bullet} 
\to S^{an}$. Base change provides canonical (homotopy coherent) comparison maps between the pullback of a pushforward on $S^{an}$ and the pushforward constructed on the refinement, which are moreover weak equivalences. 
Viewing sheaves as homotopy cartesian sections these maps give weak equivalences. 

The pushforward we have constructed preserves fibrations (as the adjoint preserves object-wise cofibrations), so it is Quillen.

Thus we may define $F^{i} \coloneqq Rf_{*}(\cat F^{i})$ and have a diagram
\[
F^{n}\to F^{n-1} \to 
\dots \to F^{0}
\]
of homotopy cartesian presections on $S_{\bullet}$ which may be interpreted as complexes of sheaves on $S$.
To see the $F^{i}$ are perfect consider the pushforward $Rf_{*}\Om^{m}_{X/S}$ on $S$. As the pushforward of a perfect complex by a smooth proper map 
this is again perfect by Lemma 2.2 of \cite{Toen12}. 

We check locally that $F^{\bullet}$ satisfies the conditions of Lemma \ref{lemma-filtbilinear}.
Let us assume $S= R\Spec(A)$ is affine and abuse notation by writing $F^{i}$ for $R\Gamma(F^{i})$.  It is clear that we can replace the $F^{i}$ by a diagram of inclusions of cofibrant objects. 
Then we need to check the following:
\begin{enumerate} 
\item All maps induce injections on $H^{*}(-\otimes^{L}_{A}H^{0}(A))$, and the $H^{*}(F^{i}\otimes^L_{A}H^{0}(A))$ are locally free sheaves.
\item There is a bilinear form $Q_{F}$.
\item The $F^{i}$ satisfy orthogonality with respect to $Q_{F}$.
 \end{enumerate} 

The first statement follows from classical Hodge theory. Writing $e_{S}$ for the inclusion $\pi^{0}(S) \to S$ we can use base change  to show that $e_{S}^{*}Rf_{*}\Om_{X/S}^{\geq p}$ is locally free because it agrees with $Rf_{*}e_X^{*}\Om_{X/S}^{\geq p} = Rf_{*}\Om_{\pi^{0}(X)/\pi^{0}(S)}^{\geq p}$, which is well-known to be locally free in classical Hodge theory. We argue similarly for the injections.

To define the shifted bilinear form we will be very explicit about our replacements. Some care is needed since fibrant cofibrant replacements are not strictly monoidal, and for technical reasons we have set up a strict orthogonality condition, rather than one up to homotopy. 

We choose an affine cover of $X$ and the injective model category of presections as a model for sheaves of $\cat O_{X}$-modules, see the discussion before Corollary \ref{cor-globalperiodstack}. 

Object-wise tensor product makes this into a symmetric monoidal model category, the pushout-product axiom holds since cofibrations are defined object-wise.
We use fibrant cofibrant replacement using functorial factorization, denote this functor by $P$. 
We also consider the transfer model category structure on algebra objects, see e.g.\ \cite{Goerss06}, and we denote fibrant cofibrant replacement with respect to this structure by $P_{alg}$. Fibrant cofibrant objects are preserved by the forgetful functor.

Next we may replace $\cat F^{i}$ by a diagram of cofibrations using functorial factorization, call it $Q'(F^{i})$, and then by a diagram of cofibrations between fibrant cofibrant objects, write this as $P'\cat F^{i}$. 
For simplicity we write $\otimes$ for the derived tensor product over the structure sheaf for the rest of this section. 
We observe that by the monoid axiom the diagram $(i,j) \mapsto Q'\cat F^{i}\otimes Q'\cat F^{j}$ is actually cofibrant in the projective model structure, and similarly for $P'$.
Thus the weak equivalences of diagrams $Q'(\cat F^{i}) \otimes Q'(\cat F^{j}) \to \cat F^{i}\otimes \cat F^{j}\leftarrow Q(\cat F^{i} \otimes \cat F^{j})$ give rise to a lift $Q'(\cat F^{i}) \otimes Q'(\cat F^{i}) \to Q(\cat F^{i}
\otimes \cat F^{j})$ as the second map is a trivial fibration.
Similarly we have functorial quasi-isomorphisms $P'(\cat F^{i}) \otimes P'(\cat F^{j}) \leftarrow Q'(\cat F^{i})\otimes Q'(\cat F^{j}) \to P(\cat F^{i} \otimes \cat F^{j})$ 
where the first map is a trivial cofibration, thus
there are compatible lifts $\eta_{ij}: P'(\cat F^{i}) \otimes P'(\cat F^{j}) \to P(\cat F^{i} \otimes \cat F^{j})$. 

We also choose a quasi-isomorphism $\epsilon: P\cat F^{0} \to P'(\cat F^{0})$.
Denote the product on $\cat F^{0}$ by $m$ and then define a multiplication on $P'\cat F^{0}$ by considering 
\[
\epsilon \circ Pm \circ \eta_{00}: P'\cat F^{0}\otimes P'\cat F^{0} \to P(\cat F^{0}\otimes \cat F^{0}) \to P(\cat F^{0}) \to P'(\cat F_{0})\]

This product vanishes on $P'\cat F^{i}\otimes P'\cat F^{n+1-i}$, as we can factor it through $\eta_{ij}$ and $m$ vanishes on $\cat F^{i}\otimes \cat F^{n+1-i}$ as $\Om^{n+1}$ is $0$. 

We push forward both $P'\cat F^{i}$ and $P_{alg}\cat F^{0}$ and functorially replace by a diagram of cofibrations again, written as $Q'$. The resulting diagram is our definition of $F^{i}$.

Now, using the lifting property for cofibrant diagrams as above and adjointness of $f_{*}$ there are natural and compatible maps 
\[\phi_{ij}: Q'f_{*}P'\cat F^{i} \otimes Q'f_{*}P'\cat F^{j} \to Q(f_{*}P'\cat F^{i} \otimes f_{*}P'\cat F^{j}) \to Qf_{*}(P'\cat F^{i} \otimes P'\cat F^{j})\]
 and we can compose $\phi_{00}$ with $Qf_{*}(\epsilon \circ Pm \circ \eta_{00})$ and a quasi-isomorphism $Qf_{*}(P'(\cat F^{0})) \to Q'f_{*}(P'(\cat F^{0}))$ to define a product on $F^{0}$ that is zero on $F^{i}\otimes F^{n+1-i}$.
 
 The construction of the bilinear form is completed by the canonical trace from Lemma \ref{lemma-constructq}.
Orthogonality for the $F^{i}$ is clear since the product of sections in $F^{i}$ and $F^{n+1-i}$ vanishes.
\end{proof}

To compare with $Rf_{*}\C$ we go back to the analytic setting.
\begin{lemma}\label{lemma-period}
Let $f: X_{T} \to T$ and $(V, Q)$ be as in the previous subsection, in particular $T$ is derived Stein. Then we have a diagram of perfect complexes over $A = N(\cat O(T))$ 
\[
F^{n}\to F^{n-1} \to 
\dots \to F^{0}  \tilde \leftarrow V \otimes A 
\] 
where 
 $F^{i} \simeq R\Ga\circ Rf_{*}(\cat F^{i})$,
the $F^{i}$ are cofibrant $A$-modules and the maps $F^{i}\to F^{i-1}$ are injective.
Moreover, the conditions of Theorem \ref{thm-periodstack} are  
satisfied, namely: 
\begin{enumerate} 
\item The last map $\bar w$ is a quasi-isomorphism. 
\item All other maps induce injections on $H^{*}(-\otimes^{L}_{A}H^{0}(A))$, and the $H^{*}(F^{i}\otimes^L_{A}H^{0}(A))$ are locally free sheaves.
\item There is a bilinear form $Q_{F}^{an}$ on $F^{0}$ which is compatible with $\bar w$ and $Q$.
\item The $F^{i}$ satisfy orthogonality with respect to $Q_{F}^{an}$.
 \end{enumerate} 
In other words we get an object of $\cat D_n(V,Q)(uT)$.
\end{lemma}
\begin{proof}
We apply the push-forward $Rf_{*}$ of sheaves from $f^{-1}\cat O^{alg}_{T}$-modules on $X_{T}$ to $\cat O_{T}^{alg}$-modules on $T$ and then global sections and consider the diagram of $R\Ga \circ Rf_{*}(\cat F^{i})$ for $i \geq 0$. These are dg-modules over $N(\cat O(T))$. 

It is clear that we can replace the $F^{i}$ by a diagram of inclusions of cofibrant objects. We first deal with some homotopy invariant properties which will not be affected by this.

To show the $F^{i}$ are perfect we apply Porta's GAGA result, see Lemma \ref{lem-gaga} below, to show that pushforward commutes with analytification, thus $Rf_{*} \Om^m_{X/T}$ is a pullback (along $h_{S}\circ i_{T}$) of the perfect complex $Rf_{*}\Om^{m}_{X/S}$, and thus perfect. 
Now $F^{i}$ is an iterated extension of perfect complexes and itself perfect as a sheaf. As it is moreover globally of finite presentation we have a perfect $\cat O(T)$-module, as follows from Lemma 8.11 of \cite{Porta14}. 

For the map in the first claim we  
use the weak equivalence $\bar w = \bar v \circ \phi$ from Lemma \ref{lemma-trivialize}. 

The second statement follows from GAGA together with Lemma \ref{lemma-filtmoduli}.

We define the shifted bilinear form using the yoga of fibrant cofibrant replacement as in Lemma \ref{lemma-filtmoduli}, letting $F^{i}$ be defined as the pushforward of the analytic Hodge filtration $Rf_{*}\Om^{\geq i}_{X/T}$.

As we are working with fibrant cofibrant objects we also have a quasi-isomorphism $F^{0}\to f_{*}P_{alg}\cat F^{0}$ and we can use Lemma \ref{lemma-trivialize} to define a second quasi-isomorphism 
$f_{*}P_{alg}\cat F^{0} \to \cat R$, where $\cat R$ is a fibrant cofibrant replacement of $Rf_{*}\C\otimes \cat O^{alg}_{T}$.
Both maps are compatible up to homotopy with the multiplication map. Denote the composition by $v$ and its homotopy inverse by $\bar v$. As $T$ is contractible there is a natural quasi-isomorphism $\phi: V \otimes A \to R\Ga(Rf_{*}\C \otimes \cat O^{alg}_{T})$. Let $\bar w$ be the composition of $\bar v$ with $\phi$, as in Lemma \ref{lemma-trivialize}.

It follows from the proof of Lemma \ref{lemma-constructq} that $\bar w$ is compatible with the shifted bilinear forms.  
\end{proof}

We needed Porta's GAGA theorem for potentially unbounded coherent sheaves in this proof. This is a slight strengthening of Theorem 7.1 in \cite{Porta15}. The following proof was
communicated to us by Mauro Porta.
\begin{lemma}\label{lem-gaga}
Let $f: X \to Y$ be a proper morphism of derived schemes and $\cat F$ a coherent sheaf on $X$. Then the natural map $(Rf_{*}\cat F)^{an} \to Rf_{*}^{an}\cat F^{an}$ is an equivalence.
\end{lemma}
\begin{proof}
One shows first that if $f^{an}_*:Coh^{\heartsuit}(X) \to Coh^+(Y)$  
has finite cohomological dimension, i.e.\ there exists $n$ such that for all $i>n$ and all $\cat F\in Coh^{\heartsuit}(X)$ we have $ R^i(f^{an}_*(\cat F))=0$, then the result holds for unbounded coherent sheaves. 

To see this one follows the proof of Theorem 7.1 in \cite{Porta15}. As in that proof one writes $\cat F$ as $\tau_{\leq n}\cat F \to \cat F \to \tau_{>n}\cat F$. Now one observes that by the result in the bounded below case $(Rf_{*}(\tau_{>k}\cat F))^{an} \simeq Rf^{an}_{*}(\tau_{> k}\cat F)^{an}$. On the other hand the assumption of finite cohomological dimension implies that the cohomology groups of $(Rf_{*}(\tau_{\leq k}\cat F))^{an}$ and $Rf^{an}_{*}(\tau_{\leq k}\cat F)^{an}$ vanish above degree $k+n$. Varying $n$ one sees that the natural map $(Rf_{*}\cat F)^{an} \to Rf_{*}^{an}\cat F^{an}$ is a quasi-isomorphism.

To show that $f^{an}$ has finite coherent cohomological dimension, for every $s \in S^{an}$ choose a Stein open neighbourhood $s\in U \subset S$ such that $\bar{U}$ is compact in $S^{an}$. Lemma 6.2 in \cite{Porta14} implies that base change $X^{an}_U 
\coloneqq U \times_{S^{an}} X^{an}$ can be covered by finitely many Stein opens. Now $X^{an}_U$ is separated (because $f$ is proper and thus separated), so the intersection of Stein spaces remains Stein. We can compute cohomology via \v Cech cohomology attached  to this Stein open cover. The open cover is finite, so the complex must be finite as well and $f^{an}$ has finite coherent cohomological dimension.
\end{proof}

\begin{rk}
We also note that the $F^{i}$ are coherent. This is proven by induction, using that
$Rf_{*}$ is exact and coherent sheaves are a stable subcategory of all sheaves. 
The crucial ingredient is that the $Rf_{*}\Om^{m}_{X/T}$ are coherent as pushforwards of coherent sheaves.

It is shown in Theorem 6.5 of \cite{Porta15} that push-forward preserves bounded below coherent complexes. As Porta explained to the authors, the same arguments as in Lemma \ref{lem-gaga} extend the result to the unbounded case if the morphism has bounded cohomological dimension.
\end{rk}
The lemma says that any $X \to T$ gives a map from $u(T)$ to the algebraic moduli stack $D_n(V,Q)$. To proceed we have to change our target to the analytification of $D_{n}(V,Q)$. We then check the map factors through the substack $U$. Note that here we need the polarization in order to define $U$. 

\begin{propn}\label{propn-localperiod}
If $T \subset S^{an}$ 
is a simply connected derived Stein space contained in the analytification of an affine subspace of $S$ then there is a derived period map \mbox{$T \to U$}, compatible with base change. 
\end{propn}
\begin{proof}
To obtain a map $T \to D_{n}^{an}$ it is enough to construct a homotopy coherent map from $T$ to the diagram $QFilt_{n}^{an} \to Perf^{an} \leftarrow *$, see Remark \ref{rk-firstdescription}. By Lemma \ref{lemma-filtmoduli} there is a map $\tilde \phi: S \to QFilt$ 
representing the filtration on $Rf_{*}\Om_{X/S}$. There is clearly a constant map $\tilde \upsilon$ representing $V \otimes \cat O_{S}$.
We may analytify both maps, compose with the natural maps to $QPerf$ and restrict to $T$ to obtain maps $\phi, \upsilon: T \to QPerf^{an}$. To complete our construction we will exhibit a homotopy equivalence $\eta$ between these maps.  As there is no higher homotopy coherence in a pullback diagram this will complete the construction.

To construct $\eta$ we consider a simplicial hypergroupoid $P_{\bullet}$ with realisation $QPerf$. As analytification commutes with homotopy colimits this will also give a presentation $P_{\bullet}^{an}$ of $QPerf^{an}$.
We will then construct a homotopy equivalence between simplicial maps $uT \to P_{\bullet}$ which gives $\eta$ using Lemma \ref{lemma-simplicial-adjoint}.

As the Tor dimension of $V$ is finite we may restrict ourselves to working with a geometric substack $QPerf^{[a,b]}$ of $QPerf$. To simplify notation we will still refer to this substack as $QPerf$.

First we construct $P_{\bullet}$. We recall that the stack of perfect complexes with finite Tor amplitude has an atlas that represents presentations of perfect complexes as iterated extensions of vector bundles, see the discussion before Lemma 3.10 in \cite{Toen07}.
Using the same methods together with the construction from Lemma \ref{lemma-perfbilinear} one sees that $QPerf$ has an atlas $U$ representing explicit quadratic forms on presentations of perfect complexes.

Then we follow the proof of Proposition 4.5 in \cite{Pridham09a} to obtain $P_{\bullet}$. In loc.\ cit.\ a simplicial affine hypercover for an arbitrary simplicial stack $Z$ is constructed.
The general construction is very delicate, but, using Pridham's notation, we are only interested in the case where the simplicial stack $Z$  is constant and thus $\nu(Z)=1$. This means that given an $r$-atlas $W$ of $QPerf$ we may define a simplicial affine presentation by $P_{i} \cong W^{\times i} \times_{QPerf^{\times i}} QPerf$. In particular $P_{1}$ represents weak equivalences between presentations of perfect complexes with non-degenerate quadratic forms.

We may assume $Rf_{*}\Om_{X/S}^{an}$ is an iteration of free vector bundles, shrinking $T$ to lie in the analytification of a suitable affine cover of $S$ if necessary. (As the Tor-dimension is finite we only have to shrink finitely often.)
Thus the complex $Rf_{*}\Om_{X/T}$ on $uT$ is represented by a map $uT \to W$, which we may view as a map from the constant simplicial affine $uT$ to $P_{\bullet}$.
We may represent $Rf_{*}\C$ similarly. Then the map $\bar w$ from Lemma \ref{lemma-period} is a weak equivalence between the presentations and thus provides a map $\Spec(\cat O(T)) \to P_{1}$, which is enough to give a homotopy between maps from the constant simplicial object $c\Spec(\cat O(T))$ to $P_{\bullet}$.

Using Lemma \ref{lemma-simplicial-adjoint}
we have the desired homotopy $T \to P_{\bullet}^{an}$ in simplicial derived Stein spaces, giving rise to $\upsilon \simeq_{\eta} \phi: T \to QPerf^{an}$, and thus to a map $T \to D_{n}(V,Q)^{an}$

Classical Hodge theory lets us check the open cohomological conditions and we see that we have a factorization $\cat P: T \to U \to D_{n}(V,Q)^{an}$. 

Compatibility with base change is clear for the maps to $QFilt$ and $*$.  Compatibility for $\eta$ is just compatibility for $\bar w$ from Lemma \ref{lemma-trivialize}.
\end{proof}

\subsection{The global derived period map}\label{sect-globalperiod}

Having constructed in the previous section a period map for any small patch $T$ of $S^{an}$, we now glue them together.

We first redo Lemma \ref{lemma-period} for a map between derived stacks. We consider perfect complexes on $S^{an}$ as fibrant cofibrant homotopy cartesian sections as in the discussion preceding Corollary \ref{cor-globalperiodstack}. 
\begin{lemma}\label{lemma-period-2}
Given $f: X^{an} \to S^{an}$ an analytification of a smooth projective map of derived Artin stacks and with $(V, Q)$ as above, we have a diagram of perfect complexes on $S^{an}$ 
\[
F^{n}\to F^{n-1} \to 
\dots \to F^{0}  \ \tilde \to \ \cat R
\] 
where $\cat R$ is a fibrant cofibrant model for $Rf_{*}\C \otimes  {\cat O}^{alg}_{S^{an}}$, the
 $F^{i}$ are fibrant and cofibrant models for $Rf_{*}(\cat F^{i})$,
 and the maps $F^{i}\to F^{i-1}$ are injective. 
Moreover, there is a bilinear form $Q_{F}^{an}$ on $F^{0}$ and the conditions of Theorem \ref{thm-periodstack} are  
satisfied.
\end{lemma}
These data look very similar to an object in $\cat D_n(V,Q)(S)$, except that we have replaced $V$ by a sheaf on $S^{an}$.  
\begin{proof} 
As in Lemma \ref{lemma-filtmoduli} we construct pushforwards as homotopy cartesian sections. 
We choose an affine cover $\{S_{i}\}_{i\in I}$ of $S$ which induces a cover $\{X^{an}_{i} = X^{an}\times_{S^{an}}S^{an}_{i}\}$ of $X^{an}$. The pushforward maps on the $X^{an}_{i}$ give a pushforward map of presections and again it follows from base change that the pushforwards are homotopy cartesian. 
To be precise, for $Rf_{*}\cat F^{i}$ we use Proposition 1.4 of \cite{Toen12}. As analytification commutes with pushforward, see Lemma \ref{lem-gaga}, base change for quasi-compact and quasi-separated maps also holds true in analytic geometry.
For $Rf_{*}\C$ we use nonabelian base change for locally compact Hausdorff spaces, see Corollary 7.3.1.18 in \cite{Lurie11a}.
As in Lemma \ref{lemma-filtmoduli} base change also implies that the pushforward construction is well-defined.

Thus we may compute $F^{i}$ as in Lemma \ref{lemma-period}, going through the same yoga of fibrant cofibrant replacement, and obtain a filtration $F^{*}$ of $F^{0}$ and a product on $F^{0}$. Similarly we compute $\cat R$. We use Lemma \ref{lemma-trivialize} we obtain a map of 
homotopy cartesian sections $\bar v: \cat R \to \cat F^{0}$. This is well defined as $\bar v$ is compatible with refinement.
This map is a weak equivalence as it is one locally, so we may find an inverse $v$ as all objects are fibrant cofibrant.
Using the product and the trace from Lemma \ref{lemma-constructq} we can define $Q_{F}^{an}$ compatibly with $v$. All the conditions may be checked locally. 
\end{proof}

We now need to combine this with the monodromy action on cohomology.
We will write $\cat V$ for the graded vector space $V = H^{*}(X_{s}, \Om^{*}_{X_{s}})$ with integral structure $V_{\Z} = H^{*}(X_{s}, \Z)$ and bilinear form $Q$, equipped with the canonical action of $\Ga = \Aut(V_{\Z},Q)$.
By abuse of notation we will also denote by $\cat V$ the corresponding sheaf on any space whose fundamental group maps to $\Ga$. We first state and prove our main theorem for the case that the base $S$ is a derived scheme.

\begin{propn}\label{propn-periodsschemes}
A polarized smooth projective morphism $f: X \to S$ of derived schemes, where $S$ is connected, quasi-separated and locally almost of finite presentation, gives rise to a derived period map $\cat P: S^{an} \to  U/{\Gamma}$, where $U$ is the derived period domain for the pair $(V,Q)$.
\end{propn}
\begin{proof} 
We will choose a suitable simplicial derived Stein space $T_{\bullet}$ with $|T_{\bullet}|\simeq S^{an}$, cf.\ the discussion at the beginning of Section \ref{sect-localperiod-1}.  Then we construct a map $uT_{\bullet} \to D_{\Gamma}$ such that the associated map to $D_{\Gamma}^{an}$ factors through $U/{\Gamma}$, using the description in Proposition \ref{propn-quotientspace}.

There is a map $\gamma: B\Ga \to QPerf$ with $\gamma^{*}\cat U = \cat V$ for the universal perfect complex $\cat U$. Here the constant stack $B\Ga$ is just the stackification of the constant presheaf $B\Ga$. 

By the proof of Proposition \ref{propn-quotientspace} it suffices to construct maps $\phi: S^{an} \to QFilt^{an}$ and $\upsilon: S^{an} \to B\Gamma^{an}$ together with a homotopy between the two compositions $S^{an} \to QPerf^{an}$. These maps will represent $Rf_{*}(\cat F^{\bullet})$ and $\cat V$, respectively.

We pick a good hypercover $T_{\bullet}$ of $S^{an}$. By this we mean $T_{\bullet} \to S^{an}$ is an effective epimorphism and if we denote by $(T_{i})_{i \in I}$ the diagram of components of $T_{\bullet}$ then all $t(T_{i})$ are contractible.

Next we show the monodromy induces a canonical map of derived analytic stacks $S^{an} \to B\Ga^{an}$. We know $\hocolim_{I} * \simeq t(S^{an})$ by the main result of \cite{Dugger04b}, thus $Rf_{*}\C$ gives a sheaf on $\hocolim_{I} *$. By formality we may write it as a graded local system classified by a map $\hocolim_{I}* \to B\Ga$ and for a model $\hocolim_{J} *$ of $B\Ga$ we may induce this by a map $I \to J$. Thus, via the maps $uT_{i}\to *$ of derived algebraic schemes, we obtain a map $uT_{i} \to B\Ga_{i}$ of simplicial derived affine schemes. Applying Lemma \ref{lemma-simplicial-adjoint}
and taking realizations we obtain: 
\[
\kappa:  
\hocolim T_{i} \to \hocolim_{I} * \to \hocolim_{J} *\to B\Ga^{an}.
\]
Note that we are comparing homotopy colimits in simplicial sets and derived stacks. But the inclusion of simplicial sets as constant stacks preserves homotopy colimits since it is a left Quillen functor into the local projective model structure on simplicial presheaves, which is a model for derived stacks.
 
We let $\gamma$ be the map $B\Gamma \to QPerf$ representing the $\Gamma$-representation $\cat V$ with its quadratic form and let $\upsilon: S^{an} \to QPerf^{an}$ be the map induced by $\gamma^{an} \circ \kappa$. 

We also define a map $\phi = (\pi \circ \tilde \phi)^{an}: S^{an} \to QFilt^{an} \to QPerf^{an}$. Here $\pi$ is just the natural forgetful map and $\tilde \phi$ is the map from Lemma \ref{lemma-filtmoduli}. 
 
 It remains to produce a homotopy $\eta$ from $\phi$ to $\upsilon$.
 
We have a comparison morphisms $v: Rf_{*}\Om^{an}_{X/S} \to Rf_{*}\C\otimes \cat O_{S^{an}}$ 
from Lemma \ref{lemma-period-2}, which we may interpret as a strict morphism of homotopy cartesian sections after pulling back to $T_{\bullet}$. On the other hand we can trivialise $Rf_{*}\C \otimes \cat O_{S}$ on $T_{\bullet}$. This gives an object in the abelian category of diagrams of chain complexes, equipped with a Lefschetz operator, so Theorem \ref{thm-deligne} provides a strict quasi-isomorphism of homotopy cartesian sections $\varphi:  H(Rf_{*}\C) \otimes \cat O^{alg}_{T_{\bullet}} \to Rf_{*}\C \otimes \cat O^{alg}_{T_{\bullet}}$. By construction there is also a quasi-isomorphism $\lambda: \cat V \otimes \cat O^{alg}_{T_{\bullet}} \to H(Rf_{*}\C) \otimes \cat O^{alg}_{T\bullet}$ which is compatible with base change as transition functions for $\cat V$ are induced by those for $H(Rf_{*}\C)$.

We may thus consider the composition of $v$ and homotopy inverses of $\varphi$ and $\lambda$ as a morphism of perfect complexes over $\cat O(T_{\bullet})$ and unravelling we obtain exactly a map $p: uT_{\bullet} \times \Delta^{1} \to P_{\bullet}$. 
(Recall from the proof of Proposition \ref{propn-localperiod} that $P_{i}$ represents presentations of $i+1$ perfect complexes with quadratic forms and explicit quasi-isomorphisms between them.)

Thus we obtain a homotopy between two maps $uT_{\bullet} \to P_{\bullet}$ in simplicial derived affine schemes. By Lemma \ref{lemma-simplicial-adjoint} this gives the desired homotopy between maps $S^{an} \simeq |T_{\bullet}| \to |P^{an}_{\bullet}| \simeq QPerf^{an}$, completing the construction of the period map.

The map is independent of our choice of $T_{\bullet}$. Recall that it suffices to check it is well-defined under refinement. 
Thus, given a smooth map
$r: T'_{\bullet} \to T_{\bullet}$ between good hypercovers we claim that the homotopy cartesian section $p': uT'_{\bullet} \times \Delta^{1} \to P_{\bullet}$ obtained by the above construction applied to $T' \to S$ is equivalent to $r^{*}p$. The map $r$ induces natural weak equivalences between the objects by the usual base change theorem. The morphism
 $\lambda$ is compatible with base change by construction. The same is true for $\varphi$ as the Lefschetz operator is compatible with $r$. For the morphism $\bar v$ we refer to Lemma \ref{lemma-period-2}. Moreover, everything is functorial with respect to maps in $T_{\bullet}$. 
\end{proof}

\begin{thm}\label{mainA}
Let $f: X \to S$ be a polarized smooth projective morphism of derived geometric stacks where $S$ is connected and of finite presentation.
Then there is a derived period map $\cat P: S^{an} \to  U/{\Gamma}$, where $U$ is the derived period domain for the pair $(V,Q)$.
\end{thm}
\begin{proof}
We can write $S \simeq \hocolim S_{i}$ for a simplicial affine scheme $S_{i}$, e.g.\ using  \cite{Pridham09a}.
Then $S^{an} \simeq \hocolim S_{i}^{an}$ since analytification of stacks is defined as a left Kan extension. Using contractible Stein hypercovers of the $S_{i}^{an}$ we may also write $S^{an} \simeq \hocolim_{j \in J} T_{j}$.

The map $f$ induces maps $f_{i}: X\times_{S}S_{i} \to S_{i}$, and since $f$ is smooth and projective so are the $f_{i}$.  
Moreover they all have homeomorphic fibres, so we can fix $V, Q$ and $n$, and thus the derived period domain $U$. 

We know that $Rf_{*}f^{-1}\cat O_{S^{an}}^{alg}$ is quasi-isomorphic to $Rf_{*}\C \otimes \cat O_{S^{an}}^{alg}$ as this is true locally by Lemma \ref{lemma-trivialize}.
It remains to show that $Rf_{*}\C$ is quasi-isomorphic to a locally constant sheaf, i.e.\ it is obtained by pullback from some sheaf on $B\Ga$.
 We have to be somewhat careful as we pass from the stack $S^{an}$ to the topological space $\hocolim t(T_{i})$ and, unlike in the case of a derived scheme, sheaves on $S^{an}$ cannot be described as sheaves on $\hocolim t(T_{j})$ in general. So we first show that $Rf_{*}\C$ is a pull-back of an infinity local system on $\hocolim_{J} t(T_{j}) \simeq \hocolim_{J} *$, and then that this infinity local system is a plain (graded) local system.

We know from Lemma \ref{lemma-trivialize} that $Rf_{*}\C$ is quasi-isomorphic to $\underline V$ on every $T_{j}$,
thus $Rf_{*}\C\otimes \cat O^{alg}_{T_{j}}$ is just the pull back of $V$ along $u(T_{j}) \to *$.
This induces transition functions 
between the $V$ which are quasi-isomorphisms. Thus we obtain a homotopy cartesian section $\cat C$ of the constant diagram indexed by $J$ that sends every $j$ to the model category of chain complexes. By strictification (e.g.\ Theorem 1 in \cite{Holstein2}) the category of these homotopy cartesian sections is equivalent to the homotopy limit of the constant diagram $J \to \dgCat$ that sends every $j$ to the dg-category $\Ch$ of (fibrant cofibrant) chain complexes. Now we recall that there is a co-action of simplicial sets on any model category, written $(K, \cat D) \mapsto \cat D^{K}$, see Chapter 16 of \cite{Hirschhorn03}.
This is a Quillen bifunctor, and thus $\holim_{J}\Ch \simeq \Ch^{\hocolim_{J}*} \simeq \Ch^{\hocolim_{J}t(T_{j})}$. (Here we use that all the $t(T_{j})$ are contractible.)
Thus $\cat C$ can be considered as an element of $\Ch^{\hocolim t(T_{j})}$, which is the category of infinity local systems on $\hocolim t(T_{j})$, see \cite{Holstein1}.
(To apply the result of \cite{Holstein2} we need to consider a diagram indexed by a direct category, without infinite ascending chains of morphisms. By Lemma 3.10 of \cite{Pridham09a} the simplicial derived scheme $S_{\bullet}$ is determined by a finite truncation. Moreover, as $S$ is quasi-compact it is enough to consider finitely many $S_{i}$. For each $i$ we can choose the hypercovers $T_j^{(i)}$ of the Stein spaces $S_{i}^{an}$ to be bounded. Restricting to non-degenerates we may thus assume that the diagram $J$ is in fact finite.)

Now $\cat C$ can equivalently be considered as a homotopy locally constant sheaf on $\hocolim t(T_{j})$, by Theorem 12 of \cite{Holstein2}. (The considerations above make sure that $\hocolim t(T_{j})$ is equivalent to a homotopy colimit of a finite diagram of points, and thus satisfies the conditions of that theorem.)

Thus $\cat C$ lives in the bounded derived category of an abelian category, namely sheaves of abelian groups on $\hocolim t(T_{j})$. We can apply Theorem \ref{thm-deligne}, since throughout all the equivalences $\cat C$ kept its Lefschetz operator. We deduce that $\cat C$ is a direct sum of its cohomology groups. So it is in fact a local system, and given by a representation $\cat V$ of the fundamental group of $\hocolim t(T_{j})$, which we may view as a sheaf pulled back via a map
\[\kappa:  
\hocolim T_{j} \to \hocolim_{J}*  \to B\Ga^{an}
\]
as in Proposition \ref{propn-periodsschemes}.
We note that all automorphisms preserve the extra structure on $V$, so the classifying map of $\cat V$ still factors through $\Ga = \Aut(V_{\Z}, Q)$. Thus $Rf_{*}\C \otimes \cat O^{alg}_{T_{\bullet}}$ is weakly equivalent to $\cat V \otimes \cat O^{alg}_{T_{\bullet}}\simeq (\gamma\circ \kappa)^{*}\cat U$ in the homotopy category. 

To complete the proof we proceed as in Proposition \ref{propn-periodsschemes}. 
We recall that Lemma \ref{lemma-filtmoduli} 
provides a map from $S$ to $QFilt$, and there is a natural map 
\[
\bar w: \cat V \otimes \cat O^{alg}_{T_{\bullet}} \to H(Rf_{*}\C) \otimes \cat O_{T_{\bullet}}^{alg} \to Rf_{*}(f^{-1}\cat O_{T_{\bullet}}^{alg}) \to Rf_{*}\Om_{X/T_{\bullet}}.
\]
 The map is a quasi-isomorphism as we can check locally by Lemma \ref{lemma-trivialize}. We may consider a homotopy inverse and obtain $w: F^{0} \to \cat V \otimes \cat O_{T_{\bullet}}^{alg}$ as a strictly compatible map of homotopy cartesian sections.
This gives a map $uT_{\bullet} \times \Delta^{1} \to P_{\bullet}$ and we conclude by Lemma \ref{lemma-simplicial-adjoint} that this is the desired homotopy. 

We observe that the map to $D_{\Ga}$ we have constructed factors through $U/\Gamma$ by classical Hodge theory.

The proof that this map is well-defined is just as in Proposition \ref{propn-periodsschemes} as we have considered compatibility with smooth base change throughout.
\end{proof}

The following example is crucial:
\begin{eg}\label{eg-moduli}
The universal example is given by taking $S$ to be the moduli stack of polarized schemes described in Example 3.39 of \cite{Pridham10}, or rather any quasi-compact component $\mathfrak M$ of its substack with smooth fibres . This is a derived 1-geometric stack and has a universal family $X \to \mathfrak M$.

It follows from \cite{Pridham10} that $\mathfrak M$ is locally almost of finite presentation, so by Theorem \ref{mainA} we have a derived period map $\mathfrak M \to U_{\Ga}$.
\end{eg}

\subsection{Comparison with the classical period map}\label{sect-comparison}

We now need to check that our construction recovers the usual period map in the underived setting. 
This is covered by the following theorem.
\begin{thm}\label{thm-trunc}
Consider $f: X \to S$ a polarized smooth projective map of derived schemes and denote the product of the classical period maps associated to $\pi_{0}\pi^{0}(f)$ by $P: \pi_{0}\pi^{0}(S^{an}) \to \pi_{0}\pi^{0}(U)$. Then $P = \pi_{0}\pi^{0}\cat P$. 
\end{thm}
\begin{proof} 
It is enough to check this locally, so let us replace $S^{an}$ by $T$ and assume $T$ is derived Stein.
We may replace the target by $D^{an}$, and we know from Theorem \ref{thm-periodstack} that
$\pi_{0}{\pi^{0}}D^{an}$ is the product of the usual closures of the period domains. 
Then we consider the map $\pi^{0}(\cat P): \pi^{0}T \to \pi_{0}\pi^{0}D^{an}$, where $\cat P$ is the derived period map.

We recall that the period map classifies the Hodge filtration on $V$, i.e.\ it is given by a map $\pi_{0}(T)$ to the (underived) analytic flag variety. We are now in the underived and non-stacky situation where the technical difficulties we were dealing with before disappear.

In particular it follows from inspecting the explicit construction that the analytic flag variety is the analytification of the flag variety $Fl$ (which is indeed a smooth variety rather than an Artin stack). Moreover the period map $\pi^{0}(T) \to Fl^{an}$ arises via an honest adjunction from $\pi^{0}(T)^{alg} \to Fl$.

We may construct the flag variety as a pull-back analogously to the derived version. The period map is then defined by $\theta: \pi_{0}(T) \to \pi_{0}\pi^{0}(QFilt)$ and a natural isomorphism $\epsilon$ between $\pi \circ \theta$ and the canonical map $V:  \pi_{0}(T) \to \pi_{0}\pi^{0}(QPerf)$.
It is clear that $\pi_{0}\pi^{0}(\tilde \phi) \cong \theta$ as it is the analytification of the algebraic classifying map. To compare $\pi_{0}\pi^{0}(\eta)$ and $\epsilon$ we note that both $\eta$ and $\epsilon$ arise via adjunction from maps classifying corresponding modules on $uT$ respectively $\pi_{0}(uT)$.

Thus we see that $\pi_{0}\pi^{0}(\cat P)$ recovers the classical period map.
\end{proof}

Together with the next section this shows that the map we construct deserves to be called the derived period map. This also shows that the underlying underived map of $\cat P$ is smooth if $S$ is smooth, but it is clear that $\cat P$ is not strong, so it is not smooth in the sense of derived analytic geometry.

\subsection{Differential of the period map}\label{sect-derivative}

In this section we will compute the differential of our period map $\cat P$
by identifying it with the differential of the infinitesimal period map of derived deformation theories considered in \cite{diNatale14, Fiorenza08, Fiorenza09}. We will not define derived deformation theory here, the reader unfamiliar with it may skip ahead to the concrete description of $d\cat P$ in Corollary \ref{cor-deriv} and its consequences.

We recall that in \cite{Fiorenza09} for a smooth projective manifold $X$ a period map is defined which goes from the derived deformation functor associated to the Kodaira-Spencer $L_{\oo}$-algebra $KS_{X}$ to derived deformations of an $L_{\oo}$-algebra $E$ associated to the de Rham complex of $X$. 
To be precise $E$ is given by the dg-Lie algebra $\uEnd(R\Ga(\Om_{X}^{*}))/\uEnd_{F}(R\Ga(\Om_{X}^{*}))[-1]$, 
We will denote this map by $\cat P_{FMM}: RDef_{KS_{X}} \to R Def_{E}$.

Then Definition 3.41 of \cite{diNatale14} presents a geometric version $\cat P_{inf}$ of this, which sends deformations of $X$ to the derived flag variety. On objects it sends a family $X \to R\Spec(B)$ over a dg-Artin algebra $B$ to the Hodge filtration on the derived de Rham complex.

The following diagram of derived deformation functors is established in \cite{diNatale14}:

\[
\bfig
\Square[RDef_{KS_{X}}`RDef_{E}`R Def_{X}`DFlag(R\Ga(X, \Om^{*}_{X}));\cat P_{FMM}`\simeq`\simeq`\cat P_{inf}]
\efig
\]

The left vertical map can be considered the inverse Kodaira-Spencer map.

Deformation functors, just like stacks, have tangent complexes, and the tangent of the map $\cat P_{FMM}$ has a very concrete description, see Corollary \ref{cor-deriv} below. We will now show this agrees (in a reasonable sense) with the differential of our derived period map $\cat P$.

Before we can make precise what we mean by agreement, let us briefly recall the tangent space in a general setting.
The tangent space of a functor $X$ at a point $x: \C \to X$ is the functor that associates to any shifted $\C$-module $M$ the space $T_{x}(M)= X(\C\oplus M)\times_{X(\C), x}*$, where $\C\oplus M$ is the square-zero extension of $\C$ by $M$. If the functor $F$ is homotopy-preserving and homotopy homogeneous then the tangent space is an abelian group and the homotopy groups satisfy  $\pi_{i}T_{x}(M)\cong \pi_{i+1}T_{x}(M[-1])$, see for example Section 1 of \cite{Pridham10a}. The tangent complex as in Section \ref{sect-tangent} thus has $i$-th cohomology given by $\pi_{0}T_{x}(\C[-i]))$.

The theory of deformations in derived analytic geometry is considered in detail in  \cite{Porta16}. Here we will only need that square zero extensions agree in the derived algebraic and the derived analytic setting, in particular that  $R\Spec(\C \oplus M) \simeq u((R\Spec(\C\oplus M))^{an})$. This follows from Proposition 8.2 of \cite{Porta15}, 
where an equivalence between derived Artin rings and (suitably defined) derived analytic Artin rings is established using the functor $(-)^{alg}$. 
As the underlying schemes are, points $(-)^{alg}$ and $u$ agree in this case.

Now let us prepare a comparison between $d\cat P$ and $d\cat P_{inf}$. 
We write $B$ for $R\Spec(\C \oplus M)$. Then to every family $X \to S$ and map $B \to S$ we can associate the deformation $\tau^{*}X \to B$ and $\cat P_{inf}$ then defines an element in $R Def_{E}(B)$, i.e.\ a map $B \to DFlag$. At the same time $\cat P$ defines an element in $D^{an}(B^{an})$ by postcomposition.

Since $\cat P$ is a derived analytic map we will compare it with the analytification of $\cat P_{inf}$.

We note that the map $\cat P_{inf}$ by definition goes to the deformation space of the derived flag variety, rather than the derived period domain, i.e. it ignores the bilinear form. Thus we replace $\cat P$ by the composition $\cat P': S \to U\to D^{an} \to DFlag^{an}$, obtained by composing the period map with the inclusion $U \to D^{an}$ and the natural forgetful map $D^{an} \to DFlag^{an}$. We note two things: As we work infinitesimally we may ignore monodromy and assume the target of the period map is $U$ rather than $U/\Gamma$. Secondly, we can consider the composition with the map $U \to DFlag^{an}$, and then we no longer need to refer to the polarization to construct the derived period map.

Thus the following result tells us that the differential of $\cat P$ is the analytification of the differential of the infinitesimal derived period map. 

\begin{propn}\label{propn-differential} 
Given $X \to S$ and $\tau: B \to S$ the maps $\cat P'\circ \tau^{an}$ and $\cat P_{inf}^{an}$ from $B^{an}$ to $DFlag^{an}$ agree and the correspondence is natural.
\end{propn}
\begin{proof}
We may view both $\cat P_{inf}^{an}$ and $\cat P'$ as constructed out of maps to $QFilt^{an}$ and $*$ with homotopies $\eta_{inf}$ and $\eta$ connecting the associated maps to $QPerf^{an}$.
The only difference is in the definition of the homotopies, which arise from maps $w_{inf}, w: \cat F^{0}\to \cat V \otimes \cat O_{B^{an}}$. The map $w_{inf}$ is obtained by analytifying an algebraic map while $w$ arises from the simplicial adjunction in Lemma \ref{lemma-simplicial-adjoint}. 
However, as the natural map $u(B^{an}) \to B$ is an equivalence the two constructions agree.
\end{proof}

\begin{propn}
The tangent space of $D^{an}$ at any point agrees with the tangent space of $D$.
\end{propn}
\begin{proof}
Given the definition of the tangent space and the definition of analytification as a left Kan extension
this follows immediately from the identification of derived analytic and derived algebraic Artin rings. 
\end{proof}
In fact, by the same argument analytification does not change any tangent spaces or induced maps between them. We can thus use the computation from \cite{diNatale14} to determine the differential of the period map.

\begin{cor}\label{cor-deriv}
At a point $s \in S$ the differential of the derived period map $d\cat P: T_{s}S \to T_{F,w,Q_{F}}D_{V,Q}$ is obtained by composing 
$\tau^{*}: T_{s}S \to T\R Def_{X_{s}}$
 with the Kodaira-Spencer map 
 \[
 T\R Def_{X_{s}} \to T\R Def_{KS_{X_{s}}} \simeq R\Ga(X, \cat T_{X})[1]\]
  and the map 
  \[
  R\Ga(X, \cat T_{X})[1] \to (\uEnd_{Q_{F}}(V)/\uEnd_{F,Q_{F}}(V))[1]
  \]
   induced by the action of $\cat T_{X}$ on $\Om_{X}$ and the product on derived global sections.
\end{cor}
\begin{proof}
This follows from the theorem as all higher cohomology is determined by evaluating the tangent functor at shifts of $\C$. Other than analytification the only difference to \cite{diNatale14} is that the second map in \cite{diNatale14} goes to $(\uEnd(V)/\uEnd_{F}(V))[1]$. But this map as constructed clearly factors through  $(\uEnd_{Q_{F}}(V)/\uEnd_{F,Q_{F}}(V))[1]$. 
\end{proof}
In other words, the differential of the derived period map can be computed, analogously to the underived version, by cup product and contraction.

This characterization of the differential immediately gives the following corollary:
\begin{cor}
The derived period map satisfies Griffiths transversality.
\end{cor}

This says that the tangent of the derived period map lies in a subcomplex of the tangent complex of the period domain, namely endomorphisms shifting the filtration by one only. Hence we could call the image of the period map a horizontal derived substack. 

\begin{rk}
At this stage we would like to say that the infinitesimal derived period map is directly obtained from the global derived period map.
Indeed, morally the map $\cat P_{inf}$ is the restriction of $\cat P$ to dg-Artin algebras.

However, to make this argument precise we would need to use a moduli stack of all varieties (as opposed to polarized varieties or subschemes of a fixed variety). But the deformation functor $R Def_{X_{0}}$ does not extend to an algebraic stack. This can be seen by considering $K3$-surfaces: The deformation functor is unobstructed with a 20-dimensional tangent space, but algebraic families of $K3$-surfaces can only be 19-dimensional. 
\end{rk}

\begin{rk}\label{rk-extended}
This is a good moment to note that our construction is still in some sense classical, and there should be a generalized (one could also say extended or non-commutative) version of the period map that starts from $R\Ga(X, \wedge^{*}\cat T_{X})$, which classifies $A_{\oo}$-deformations of $D^{b}_{coh}(X)$. This appears in work by Kontsevich and Barannikov \cite{Kontsevich03, Barannikov98}. 
To take this further, one would need to understand non-commutative Hodge structures, see \cite{Katzarkov08a}.

The reader may also consider \cite{Fiorenza08} for an explicit period map for generalized deformations of a K\"ahler manifold and \cite{Iwanari16} for a period map for deformations of a non-commutative algebra. 
\end{rk}

\subsection{Examples}\label{sect-examples}

We will finish by briefly talking about some examples for the derived period map. 

Recall that we already described the universal derived period map defined on substacks of the moduli space $\mathfrak M$ of polarized schemes in Example \ref{eg-moduli}. 

\begin{eg}
Another class of examples is given by letting $S$ be a subscheme of the derived Hilbert scheme $DHilb_{Y}$ for some projective variety $Y$. This was constructed originally in \cite{Ciocan02} and appears in a more modern context in \cite{Lurie04a}, see also Section 3.3 of \cite{Pridham10}. $DHilb_{Y}$  is a geometric stack locally almost of finite presentation that
parametrizes derived families of subschemes of $Y$, to be precise it represents the functor sending  $S$ to derived $Y\times S$-schemes which are proper and flat over $S$ and for which the map to $Y\times S$ is a closed immersion locally almost of finite presentation. 
\end{eg}

In both of these examples it is useful to observe the following: Fix a smooth scheme $X_{0}$ or a smooth subscheme $X_{0}$ of some projective $Y$. 
Then there is an open substack $S$ of $\mathfrak M$ respectively $DHilb_{Y}$ containing $X_{0}$ such that the universal family $f: X \to S$ is smooth and projective. (The underived truncation of $f$ being smooth is an open condition 
and the condition that $f$ is strong is again open as it says a certain map of graded modules on $\pi^{0}(S)$ is an isomorphism.) To apply the main theorem we just need to make sure $S$ is quasi-compact, restricting to some substack if necessary. (Note that this is equivalent to the underlying underived space of $S$ being quasi-compact, see Lemma \ref{lemma-open}.)

The universal families on $\mathfrak M$ and $DHilb_{Y}$ give smooth projective families of derived schemes $X \to S$ by restriction and thus derived period maps $\cat P: S^{an} \to U/\Ga$.

Let us now look infinitesimally at two concrete examples. We note that the tangent complex for the moduli stack $\mathfrak M$ of polarized schemes at $X$ is an extension of $\cat T_{X}[1]$ by $\cat O_{X}[1]$, see Example 3.39 in \cite{Pridham10}. The tangent complex of the Hilbert scheme $DHilb_{Y}$ at $X \subset Y$ is $\set L^{*}_{X/Y}$, if $X$ is smooth this is just the normal bundle.

Thus the cases where $H^{i}(\cat T_{X_{s}})$ is nonzero for $i > 1$ will often correspond to interesting derived information in $S$.

\begin{eg}
If $X_{0}$ is any Calabi-Yau variety then the infinitesimal derived period map induces an injection on the cohomology groups of tangent complexes, this follows from 
Theorem B of \cite{Iacono09}. 

We can use this to write down examples where the derived period map is a non-trivial enhancement of the usual period map.
Let us for example consider an abelian surface $X$ in a component $\mathfrak M$ of the moduli stack of polarized varieties. Then $X$ has nontrivial $H^{2}(\cat T_{X})$, and there is a surjection from $\pi_{1}(T_{X}{\mathfrak M})$ to $H^{2}(\cat T_{X})$. Thus the derived period map induces a non-trivial map in degree 1 of the tangent complex at $X$.
\end{eg}

We note that Theorem A of \cite{Iacono09} says that whenever the period map induces an injection on the cohomology of tangent spaces then the deformation theory of $X_{0}$ is unobstructed, as the deformation theory of a filtered complex is governed by a quasi-abelian $L_{\oo}$-algebra.
(In fact generalized deformations of Calabi-Yau varieties are still unobstructed, see \cite{Katzarkov08a}.)

In particular the derived period map only sees unobstructed deformations of the special fibre, and the kernel of the differential of the derived period map is a natural global reduced obstruction theory.

One question arising immediately from the existence of the period map is when it is injective in a suitable sense, i.e.\ when a moduli problem can be completely embedded in the period domain (which has a nice description in terms of linear and quadratic data).

One possible way of making this precise would be to ask when $\pi_{0}\pi^{0}(\cat P)$ is an immersion and $d\cat P$ is injective on homotopy groups.
 
\begin{rk}
We note that this is not the definition of a closed immersion in the sense of \cite{Toen05a} (which does not constrain the homotopy groups of the domain), nor is it a strong map (which constrains the homotopy groups of the domain too much). This is an example of the lack of a good derived analogue of closed immersion. 
\end{rk}

This is a kind of \emph{derived Torelli} problem. (The reader should note that this term is already used for the problem of determining when the period map determines the derived category of a variety.)

\begin{eg}
Now consider the case where $X_{0}$ is a hypersurface of degree $d$ in $\set P^{3}$. We may consider it in the moduli stack $\mathfrak M$ of polarized surfaces. Then derived Torelli is false if $d$ is large enough, i.e.\ the map on homotopy groups of tangent complexes is not an injection. 
It follows from our tangent space considerations that derived deformations of $X_{0}$ in degree 1 surject to $H^{2}(\cat T_{X_{0}})$.
One can compute that  $H^{2}(\cat T_{X_{0}})$ is nonzero (using Kodaira vanishing, Hirzebruch-Riemann-Roch and the normal exact sequence). But as $H^{*}(X_{0},\Om^{*})$ is concentrated in even degrees the infinitesimal period map must send $H^{2}(\cat T_{X_{0}})$ to zero.

This is particularly interesting since projective hypersurfaces of degree bigger than three satisfy the local Torelli theorem, see \cite{Carlson03}, thus $\pi_{0}\pi^{0}(\cat P)$ is locally an immersion.
\end{eg}

\appendix
\section{Presenting higher analytic stacks}\label{appendix}
\subsection{Introduction} 

In this appendix we show that Pridham's framework of presenting higher stacks as hypergroupoids developed in \cite{Pridham09a} can be used to provide models for derived analytic Artin stacks. For an overview of the general theory see \cite{Pridham11}. While some of this discussion is not strictly necessary and only Lemma \ref{lemma-simplicial-adjoint} is needed in the body of the paper we feel that this discussion provides an interesting and useful viewpoint. 

\subsection{Hypergroupoids}
In \cite{Pridham09a} Pridham develops the theory of hyper{-}groupoids as presentations of higher stacks.

To talk about hypergroupoids we need to fix a model category $\cat S$ with a nice subcategory $\cat A$ and classes $\mathbf C$ and $\epsilon$ of morphisms in $Ho(\cat A)$ and $Ho(\cat S)$ respectively, satisfying some conditions we will detail below.

\begin{eg}
The motivating example is when $\cat S$ is the model category of stacks on derived affine schemes over some ground ring $k$ with the \'etale topology.
$\cat A$ consists of the essential image of the Yoneda embedding, $\epsilon$ is the class of local surjections in $Ho(\cat S)$ and $\mathbf C$ the class of smooth maps in $Ho(\cat A)$ which are moreover in $\epsilon$. We write this quadruple as $(\cat A_{alg}, \cat S_{alg}, \mathbf C_{alg}, \epsilon_{alg})$
\end{eg}

We write $s\cat A$ for $\cat A^{\Delta^{op}}$.

\begin{defn}
A map $X \to Y$ in $s\cat A$ is a \emph{relative $(n,\mathbf C)$-hypergroupoid} over $Y$ if it is a Reedy fibration and the homotopy partial matching maps
\[
X_{m} \to M^{h}_{\Lambda^{m}_{k}}(X)\times^{h}_{M^{h}_{\Lambda^{m}_{k}}(Y)}Y_{m}
\]
are in $\mathbf C$ for all $m,k$, and weak equivalences if $m > n$. A relative hypergroupoid over the final object is simply called a \emph{$(n, \mathbf C)$-hypergroupoid}.
\end{defn}
\begin{defn}
A map $X \to Y$ in $s\cat A$ is a \emph{trivial relative $(n,\mathbf C)$-hypergroupoid} over $Y$ if it is a Reedy fibration and the homotopy matching maps
\[
X_{m} \to M^{h}_{\partial \Delta^{m}}(X)\times^{h}_{M^{h}_{\partial\Delta^{m}}(Y)}Y_{m}
\]
are in $\mathbf C$ for all $m$, and weak equivalences if $m \geq n$. 

We will fix  $n$ in this section and simply speak of hypergroupoids when $\mathbf C$ is clear from the context. 

\end{defn}
Here the matching objects functor $M^{h}_{-}(X): \sSet^{op} \to \cat A$ is defined by forming the right Kan extension of the functor $X: \Delta^{\op} \to \cat A$, see Section 1.1.1 of \cite{Pridham09a}. It is the derived version of $\Hom_{\sSet}(-,X)$.

Now the geometric realization functor $|-|$ from $s\cat A$ to $\cat S$ sends hypergroupoids to geometric stacks, see Proposition 4.1 of \cite{Pridham09a}. (This is just another name for the homotopy colimit over $\Delta^{op}$.) For simplicity we will always compose $|-|$ with stackification (i.e.\ fibrant replacement in $\cat S$) without mentioning it in the notation. 

Theorem 3.3 in \cite{Pridham11} says that in the example above the relative category of $(n,\mathbf C_{alg})$-hypergroupoids with weak equivalences given by trivial hypergroupoids gives a model for the $\oo$-category of strongly quasi-compact $n$-geometric derived Artin stacks.

\begin{rk}
Recall that a \emph{relative category} \cite{Barwick12} is just a category with a collection of weak equivalences satisfying some very basic conditions. Nevertheless, the model category of relative categories is Quillen equivalent to other models of $\oo$-categories (i.e.\ simplicial categories, quasi-categories, Complete Segal Spaces).
\end{rk}

We will denote the level-wise hom-space for $s\cat A$ by $\uHom$, but it is important to note that we cannot compute hom-spaces level-wise. Instead we define
\[
\uHom^{\#}_{s\cat A}(X,Y) = \uHom_{pro(s\cat A)}(\tilde X,Y)
\]
where $\tilde X \to X$ is a $\cat T$-projective relative $\cat T$-cocell where $\cat T$ is the class of trivial relative hypergroupoids. A relative $\cat T$-cocell is just a transfinite composition of pullbacks of maps in $\cat T$. For detailed definitions see Section 3.2 of \cite{Pridham09a}. 
To understand this definition it helps to note that $\pi_{0}(\uHom^{\#}(X,Y) = \lim_{X' \in \cat T(X)}\pi_{0}\uHom(X', Y)$, where the limit is over trivial relative hypergroupoids over $X$.

Then if $Y$ is a hypergroupoid we have that $\uHom^{\#}_{s\cat A}(X,Y) \simeq \Map_{\cat S}(|X|, |Y|)$, see Theorem 4.10 in \cite{Pridham09a}.

We want to apply this framework to study analytic stacks. Thus we need a new choice of $(\cat A, \cat S, \mathbf C, \epsilon)$.

Recall that $dStein$ is the full subcategory of $dAn$ given by those derived analytic spaces whose underlying underived analytic space is Stein. This is constructed as a quasi-category in \cite{Porta15}.

We need $\cat A$ to be a pseudo-model category. To be precise we need a model category $\cat S$,  a class of morphisms $W_{\cat A}$ in $\cat A$ and a fully faithful $\iota: \cat A \to \cat S$ such that
\begin{itemize}  
\item $\iota(W_{\cat A})= W_{\cat S} \cap \iota(\cat A)$. Here $W_{\cat S}$ are the weak equivalences in $\cat S$.
\item $\cat A$ closed under weak equivalences in $\cat S$, i.e.\ $x \simeq \iota(y)$ implies $x$ is in the image of $\cat A$.
\item $\cat A$ is closed under homotopy pullbacks in $\cat S$
\end{itemize}

Consider the quasi-category $dStein$ and recall that simplicial categories and quasi-categories are equivalent models of $\oo$-categories. We apply $\mathfrak C$, the left adjoint of the coherent nerve $N$, to obtain a simplicial category $\mathfrak C[dStein]$. 
We consider $\sSet^{\mathfrak C[dStein])}$, the simplicial category of functors from $\mathfrak C[dStein]$ to the simplicial model category of simplicial sets, $\sSet$. 
We equip this category with the injective model structure. 
There is a Grothendieck topology $\tau$ on $dStein$ induced by \'etale morphisms of $\cat T_{an}$-structured topoi (see Definition 2.3.1 of \cite{Lurie11e}). 
We localize $\sSet^{\mathfrak C[dStein]}$ at homotopy $\tau$-hypercovers to obtain the model category of stacks, which we call $\cat S_{an}$. 
As $\tau$ is (hyper-)subcanonical, see Corollary 3.6 in \cite{Porta15}, we have a fully faithful embedding $y: \mathfrak C[dStein] \to \cat S_{an}$.

We let $\cat A_{an}$ be the closure of $y(\mathfrak C[dStein])$ under weak equivalences and define $W_{\cat A}$ by the first condition above. 

Thus the first two conditions are satisfied. For the third condition we need to know that $dStein$ is closed under homotopy pullbacks. By Section 12 of \cite{Lurie11d} $dAn$ has (homotopy) pullbacks. 
Now note that $\pi^{0}$ commutes with limits and the fibre products of Stein spaces are Stein spaces, see Example 51 (b) in \cite{Remmert94}.

\begin{rk}
For the reader's peace of mind we note that working with simplicial presheaves as simplicial functors is equivalent to working with presheaves defined in the quasi-categorical setting as in \cite{Porta14}. 
This follows from Proposition 4.2.4.4 of \cite{Lurie11a}. Setting  $\mathcal U = A = \sSet$ we have $N((\sSet^{\mathfrak C[dStein]})^{\circ}) \simeq Fun(dStein, N(\sSet^{\circ}))$ where $N$ is the coherent nerve. Here $(-)^{\circ}$ restricts to fibrant cofibrant objects in a model category. Of course $N(\sSet^{\circ})$ is the infinity category of simplicial sets. In other words the simplicial category $(\sSet^{\mathfrak C[dStein]})^{\circ}$ models the quasi-category of functors from $dStein$ to simplicial sets. This is true for the injective or projective model structure.

The reason we are working not with quasi-categories but simplicially with $\sSet^{\mathfrak C[dStein]}$ is that we rely on explicit arguments involving simplicial diagrams in Theorem \ref{thm-context} below.
\end{rk}

Having made these observations we will abuse notation and use $dStein$ for $\mathfrak C[dStein]$ from now on.

Next, given the pseudo-model category $\cat A_{an} \subset \cat S_{an}$ we need a class of $\epsilon$-morphisms in $\cat S_{an}$, functioning as covers, closed under composition and homotopy pullback. They need to satisfy Properties 1.7 of \cite{Pridham09a}.

We just imitate the algebraic definition and let $\epsilon_{an}$ be the class of local surjections, i.e.\ maps such that the associated map of simplicial sheaves is surjective on $\pi_{0}$.

Next we need a class $\mathbf C_{an}$ of morphisms in $Ho(\cat A_{an})$ containing isomorphisms and closed under composition and homotopy pullback, and satisfying Properties 1.8 of \cite{Pridham09a}. We define $\mathbf C_{an}$ to consist of smooth maps which are also in $\epsilon_{an}$.

\begin{thm}\label{thm-context}
$(\cat A, \cat S_{an}, \mathbf C_{an}, \epsilon_{an})$ as above satisfies Properties 1.7 and 1.8 of \cite{Pridham09a}.
\end{thm}
\begin{proof}
We just follow the proof of Proposition 1.19 of \cite{Pridham09a}.

1.8 is true for any simplicial site. For (1) see Proposition 3.1.4 in \cite{Toen04}, (2) holds by definition and (3) is clear as smooth maps are defined locally.

For 1.7 we use $dStein$ in place of the category $\cat T$ in loc.\ cit.\ and then consider simplicial objects in $\cat I \coloneqq \cat S_{an} = sPr_{\tau, inj}(dStein)$, which is a category of simplicial presheaves with the local injective model structure on a simplicial site, thus all the computations involved are entirely unchanged.
\end{proof}
\begin{rk}
The last ingredient we need is Assumption 3.20 of \cite{Pridham09a}. However, we can always satisfy the condition by choosing two universes, see Remark 3.21 in loc.\ cit.
\end{rk}

Note that we may replace objects in $s\cat A_{an}$ by objects in $dStein^{\Delta^{op}}$ and talk about genuine simplicial derived Stein spaces. This is shown in Section 4.4 of \cite{Pridham09a}, the main input being the proof of Lemma 1.3.2.9 of \cite{Toen05a}, which goes through in the analytic setting.

We can now consider the category of derived analytic stacks as modelled by hypergroupoids. As all the results in \cite{Pridham09a} hold for analytic stacks we can deduce the following:
\begin{thm}
The relative category of $(n, \mathbf C_{an})$-hypergroupoids, with weak equivalences given by relative trivial hypergroupoids, is a model for the $\oo$-category of 
strongly quasi-compact $n$-geometric derived analytic Artin stacks.
\end{thm}
Here we note that derived Artin stacks are the subcategory of derived analytic stacks $\cat S_{an}$, given by the usual representability conditions, cf.\ Definition 2.8 of \cite{Porta14}.

\begin{rk}
Note that there is a difference in definition between a geometric context in the sense of \cite{Porta14}, like $(dStein, \tau, P)$, and a HA-context in the sense of \cite{Toen05a}, which is cited in \cite{Pridham09a}.
But the definition of geometricity is the same in both contexts.
\end{rk}

\subsection{Analytification of hypergroupoids}

The derived analytification functor $an: X \mapsto X^{an}$ between structured $\oo$-topoi sends derived affine schemes 
to $\cat T_{an}$-affines, see Proposition 2.3.18 of \cite{Lurie11e}.
In particular derived affine schemes locally almost of finite presentation are sent to derived Stein spaces. (Analytification maps $\cat T_{\acute et}$-schemes to $\cat T_{an}$-schemes, and the truncation is sent to a Stein space.) We write $dAff^{lfp}$ for derived affine schemes locally almost of finite presentation. This is just the opposite of the category of homotopically finitely presented simplicial algebras.

We now want to consider a partial left adjoint to this functor and define $u = R\Spec(\cat O(-))$ on $dStein$. 
This is an affinized forgetful or algebraization functor. It is clear that $u$ sends affines to affines. 

\begin{lemma}\label{lemma-affine-adjoint}
There is a natural weak equivalence $\Map_{dStein}(T, Y^{an}) \simeq \Map_{dAff}(uT, Y)$ for a derived Stein spaces $T$ and a derived affine scheme $Y$ locally almost of finite presentation.
\end{lemma}

\begin{proof}
We consider the functor $R\Spec(\Gamma(-)) \circ (-)^{alg}$ between the $\oo$-categories $Top(\cat T_{an})$ and $dAff$. This extends the functor $u$ we are considering. Moreover, as a composition of right adjoints it has a left adjoint: $dAff \to Top(\cat T_{an})$ given by inclusion followed by analytification. Now we note that $an$ sends $dAff^{lfp}$ to $dStein$ and conclude. 
\end{proof}
\begin{rk}
There is a slightly more concrete way of seeing this weak equivalence: 
The correspondence is clear if $Y = \set A^{1}$, as in this case both sides are just global functions on $T$. Then we can extend to all derived affine schemes by observing that $an$ preserves limits, and derived affine schemes are generated under limits by $\set A^{1}$.
\end{rk}

\begin{rk}
In fact we can extend $\Map(T, Y^{an}) \simeq \Map(uT,Y)$ to the case where $T$ is a derived analytic space which is a colimit of derived Stein spaces. We can thus extend to simplicial derived Stein spaces, but note that $u$ is not going to preserve hypergroupoids.
\end{rk}

As we may represent derived algebraic and analytic stacks by simplicial derived affine schemes resp.\ simplicial derived Stein spaces the following lemma is useful. We can then define a level-wise analytification $an$ for simplicial derived affines.

 \begin{lemma}\label{lemma-anpreserves}
Level-wise analytification sends (trivial) relative hypergroupoids in $dAff^{lfp}$ to (trivial) relative hypergroupoids.
\end{lemma}
\begin{proof}
By Lemma \ref{lemma-affine-adjoint} $an$ preserves homotopy limits in $dAff^{lfp}$. Moreover it sends smooth maps to smooth maps and preserves local surjections.  It follows that it preserves relative hypergroupoids.
\end{proof}
Moreover, level-wise analytification agrees with the Lurie-Porta definition that we quoted in Section \ref{sect-dan}. This also shows it is well-defined under weak equivalences of hypergroupoids.
\begin{lemma}\label{lemma-anwell}
Let $Y$ be a hypergroupoid in $dAff^{lfp}$. Then $|an(Y)|_{\cat S_{an}} = (|Y|_{\cat S_{alg}})^{an}$.
\end{lemma}
\begin{proof}
 This just says that analytification commutes with geometric realization. This is clear as analytification of Artin stacks is defined as a left Kan extension.
 \end{proof}
 
However, we note that there is no well-defined algebraization functor from analytic hypergroupoids to algebraic hypergroupoids, contrary to claims in previous versions of this paper. For a systematic way to get around this difficulty, see \cite{HolsteinF}. 

 We can now state a way of producing maps into analytifications.
 \begin{lemma}\label{lemma-simplicial-adjoint}
For a simplicial derived Stein space $T_{\bullet}$ and a  derived algebraic stack $Y$ locally almost of finite presentation with simplicial presentation $Y_{\bullet}$
there is a natural map $\Map_{s\cat A_{alg}}(uT_{\bullet}, Y_{\bullet}) \to \Map_{s\cat A_{an}}(T_{\bullet}, an(Y_{\bullet})) \simeq \Map_{dAnSt}(|T_{\bullet}|, Y^{an})$.
\end{lemma}
\begin{proof}
The first arrow is immediate from Lemma \ref{lemma-affine-adjoint} and applying the realization functor. 
The second map comes from Lemma \ref{lemma-anwell}. 
\end{proof}
\begin{rk}
The map in the lemma is not well-defined on the level of derived stacks and it appears that this shortcoming cannot be overcome. This means that whenever we want to apply this lemma for some cover $T_{\bullet}$ of a derived analytic stack, we need to check by hand that the map we construct does not depend on our choice of $T_{\bullet}$.
\end{rk}

\bibliography{../biblibrary2}

\end{document}